\documentclass[final,reqno]{amsart}
\usepackage{amsmath,amsthm,amsfonts,amscd,amssymb}
\usepackage[hmargin={28mm,28mm},vmargin={30mm,35mm}]{geometry}
\usepackage{array}
\usepackage{enumerate} 
\usepackage{dsfont}
\usepackage{graphicx}
\usepackage{tikz}
\usepackage{xspace}
\usepackage[ocgcolorlinks,allcolors={blue},breaklinks]{hyperref}
\usepackage{xcolor}
 \usepackage[normalem]{ulem}
 \newcounter{corr}
 \definecolor{violet}{rgb}{0.580,0.,0.827}
 \newcommand{\corr}[3]{\typeout{Warning : a correction remains in page
 \thepage}
 				\stepcounter{corr}        
 				{\color{blue}\ifmmode\text{\,\sout{\ensuremath{#1}}\,}\else\sout{#1}\fi}
         {\color{red}#2}
         {\color{violet} #3}}
\numberwithin{equation}{section}
\newtheorem{theorem}{Theorem}
\newtheorem{lemma}[theorem]{Lemma}

\theoremstyle{remark}
\newtheorem{remark}[theorem]{Remark}
\theoremstyle{definition}

\newtheorem{definition}[theorem]{Definition}

\newcommand{\Real}{\mathbb{R}}
\newcommand{\Natural}{\mathbb{N}}
\newcommand{\norme}[1]{\left\Vert #1\right\Vert}

\newcommand{\ug}{\boldsymbol{u}}
\newcommand{\vg}{\boldsymbol{v}}
\newcommand{\xg}{\boldsymbol{x}}
\newcommand{\dd}{\mathrm{d}}

\newcommand{\DIV}{\text{div }}
\newcommand{\wg}{\boldsymbol{w}}

\newcommand{\normal}{\mathbf{n}}
\newcommand{\BS}[1]{\boldsymbol{#1}}
\newcommand{\CM}{\Omega}
\newcommand{\CMT}{\widehat{\Omega}}

\newcommand{\etap}{\overset{.}{\eta}}
\newcommand{\hb}{{\underline{h}}}

\title[]{{On an existence theory for a fluid-beam problem encompassing possible contacts}}
\author{Jean-J\'er\^ome Casanova}
\address[Jean-J\'er\^ome Casanova]{CEREMADE, UMR CNRS 7534, Universit\'e Paris-Dauphine, PSL Research University, Place du Mar\'echal de Lattre de Tassigny, 75775 Paris Cedex 16, France}
\email{casanova@ceremade.dauphine.fr}

\author{C\'eline Grandmont}
\address[C\'eline Grandmont]{Inria Paris, 75012 Paris, France \&  Sorbonne Universit\'e, UMR 7598 LJLL,75005 Paris, France}
\email{celine.grandmont@inria.fr}

\author{Matthieu Hillairet}
\address[Matthieu Hillairet]{IMAG, Univ Montpellier, CNRS, Montpellier, France}
\email{matthieu.hillairet@umontpellier.fr}

\begin{document}

\begin{abstract}
In this paper we consider a coupled system of pdes modelling the interaction between a two--dimensional incompressible viscous fluid and a one--dimensional elastic beam located on the upper part of the fluid domain boundary. We design a functional framework to define weak solutions in case of contact between the elastic beam and the bottom of the fluid cavity. We then prove that such solutions exist globally in time regardless a possible contact by approximating the beam equation by a damped beam and letting this additional viscosity vanishes.
\end{abstract}

\maketitle

\section{Introduction}

In this paper we consider a fluid--structure system coupling a $2$D homogeneous viscous incompressible fluid with a $1$D 
elastic structure. When the elastic structure is at rest, the fluid domain is of rectangular type and the structure is located on the upper part of the fluid domain boundary.  The fluid is described by the Navier--Stokes equations set in an unknown domain depending on the structure displacement that is assumed to be only transverse and that satisfies a beam equation.  Since the fluid is viscous it sticks to the boundaries so that the fluid and the structure velocities are equal at the interface. Finally, the fluid applies a surface force on the structure. Such coupled nonlinear models  can be viewed as toy models to describe the blood flow through large arteries.

The existence of a solution to the Cauchy problem associated with this kind of systems has been intensively studied in the last years. In \cite{DaVeiga, Lequeurre11, Lequeurre13} existence and uniqueness of a strong solution locally in time is proved in  case  additional viscosity is added to the structure equation (so that the structure displacement satisfies a damped Euler--Bernoulli equation). When no viscosity is added and in  case the dynamics of the structure displacement is governed by a membrane equation, existence and uniqueness of a local strong solution is obtained in \cite{Grandmont-Hillairet-Lequeurre}. The beam case with no additional viscosity is investigated in \cite{Badra-Takahashi}, where existence of strong solution locally in time (or for small data) is proved but with a gap between the regularity of the initial conditions and the propagated regularity of the structure displacement.  Existence of weak solutions is obtained in \cite{Chambolle-etal} for $3$D-$2$D coupling where the structure behaviour is described by a viscous plate equation and in \cite{Grandmont08, Muha-Canic13} in the non-viscous case. Let us also mention weak existence results on fluid--shell models \cite{Lengeler, Muha-Schwarzacher}. Note that these results are obtained as long as the structure does not touch the bottom of the fluid cavity (or, in  case of shells, as long as there is no self contact). More recently, in \cite{Grandmont-Hillairet}, the authors establish  existence of a global-in-time strong solution in the $2$D-$1$D case when the structure is governed by a damped Euler--Bernoulli equation. This global-in-time result is a consequence of a no contact one: it is proven therein that, for any $T>0,$ the structure does not touch the bottom of the cavity. The proof of this latter result relies strongly  on the additional viscosity in the beam equation and on the control of the curvature of the structure. 

\medskip

The question we address here is: can we prove existence of a global weak solution regardless of a possible contact (for an undamped beam)? We aim to take advantage of the existence of global strong solution for a viscous structure and let the additional viscosity tend to zero.  Our scheme is inspired by the one developed in \cite{SanMartin-Starovoitov-Tucsnak} where the global existence of a weak solution is derived for a $2$D fluid--solid coupled problem. However, in \cite{SanMartin-Starovoitov-Tucsnak} the solids are viewed as  inclusions whose viscosities is infinite. The fluid--solid problem is then approximated by a completely fluid problem with different viscosities in the inclusions and in the fluid. The viscosity of the inclusions is then sent to infinity. In contrast, in our case the parabolic--hyperbolic fluid--structure system is approximated by a parabolic--parabolic one by adding viscosity to the structure. We prove that, up to the extraction of a subsequence, the sequence of solutions of the damped  system converges towards a weak solution (in a sense to be defined) of the undamped system. The main difficulties are to define functional and variational frameworks compatible with a possible contact and to prove the strong compactness of the velocities, also in  case of a possible contact. Indeed the proof developed for instance in \cite{Grandmont08}, where the vanishing viscosity limit is also studied, strongly relies on the fact that the elastic structure does not touch the bottom of the fluid cavity. 

\subsection{The fluid-structure model}
We  introduce now the damped coupled fluid--structure system. We refer to
this system as $(FS)_{\gamma}$, where the subscript $\gamma$ is used to track the dependency with respect to the ``viscosity'' of the structure. The configuration ``at rest" of the fluid--structure system is assumed to be of the form $(0, L)\times (0, 1)$ where the elastic structure occupies the part of the boundary $(0, L)\times \{1\}$. The deformed fluid set is denoted by $\mathcal{F}_{h_{\gamma}(t)}$. It depends on the structure vertical deformation $h_{\gamma}=1+\eta_\gamma$, where $\eta_\gamma$ denotes the elastic vertical displacement. 
Thus, the deformed fluid configuration reads:
\begin{equation}\label{fluid.domain}
\mathcal{F}_{h_{\gamma}(t)}=\{(x,y)\in\mathbb{R}^{2}\mid 0<x<L,\,0<y<h_{\gamma}(x,t)\}.
\end{equation}
 The deformed elastic configuration is denoted by $\Gamma_{h_{\gamma}(t)}=\{(x,y)\in\Real^{2}\mid x\in(0,L),\,y=h_{\gamma}(x,t)\}$.
 
\medskip
 
 The fluid velocity  $\ug_{\gamma}$ and the fluid pressure $p_{\gamma}$ satisfy the $2$-D incompressible Navier--Stokes equations in the fluid domain:
\begin{equation}\label{Navier-Stokes}
\begin{aligned}
\rho_{f}(\partial_{t}\ug_{\gamma} + (\ug_{\gamma}\cdot\nabla)\ug_{\gamma}) - \DIV{\sigma(\ug_{\gamma},p_{\gamma})}  &{}= 0\text{ in }\mathcal{F}_{h_{\gamma}(t)},\\
\DIV{\ug_{\gamma}}&{}=0\text{ in }\mathcal{F}_{h_{\gamma}(t)},
\end{aligned}
\end{equation}
where $\sigma(\ug_{\gamma},p_{\gamma})$ denotes the fluid stress tensor given by the Newton law:
\[\sigma(\ug_{\gamma},p_{\gamma})=\mu(\nabla\ug_{\gamma} +( \nabla\ug_{\gamma})^{T}) -p_{\gamma}I_{2}.\]
In the previous equations $\rho_{f}>0$ and $\mu>0$ are respectively the fluid density and viscosity. The structure displacement $\eta_{\gamma}$ satisfies a damped Euler--Bernoulli beam equation:
\begin{equation}\label{Beam.equation}
\rho_{s}\partial_{tt}\eta_{\gamma}-\beta\partial_{xx}\eta_{\gamma}-\gamma\partial_{xx}\partial_{t}\eta_{\gamma}+\alpha\partial_{x}^4\eta_{\gamma}=\phi(\ug_{\gamma},p_{\gamma},\eta_{\gamma})\text{ on }(0,L).
\end{equation}
The constant $\rho_{s}>0$ denotes the structure density and $\alpha,\beta,\gamma$ are non negative parameters.
Through this paper we assume that $\alpha>0$. 
This restriction guarantees sufficient regularity of the structure deformation in the compactness argument. The reader shall note for instance that we need $H^{1+\kappa} \cap W^{1,\infty}$ regularity of the deformation in Lemma \ref{lem:projector}.

\medskip

The source term $\phi$  in the right-hand side of the beam equation arises from the action--reaction principle between the fluid and the structure. It represents the force applied by the fluid on the structure. It can be defined by the variational identity
\begin{equation}
\label{def:forcefluide}
\int_{0}^{L}\phi(\ug_{\gamma},p_{\gamma},\eta_{\gamma})\cdot \varphi(x,h_{\gamma}(x))\BS{e}_{2}\dd x=\int_{\Gamma_{h_{\gamma}(t)}}\sigma(\ug_{\gamma},p_{\gamma})\normal_{\gamma}\cdot \varphi_{\vert \Gamma_{h_{\gamma}(t)}}\BS{e}_{2}\dd \Gamma_{h_{\gamma}(t)},
\end{equation}
for any regular test function $\varphi$ and where $\displaystyle\normal_{\gamma}$ denotes the unit exterior normal to the deformed interface: 
\[
\normal_{\gamma}=\frac{1}{\sqrt{1+(\partial_{x}\eta_{\gamma})^{2}}}\begin{pmatrix}
-\partial_{x}\eta_{\gamma}\\
1
\end{pmatrix}.
\] 
Since the fluid is viscous the following kinematic condition holds true at the interface 
\begin{equation}\label{interface.condition}
\ug_{\gamma}(x,h_\gamma(x,t),t)=\partial_{t}\eta_{\gamma}(x,t)\BS{e}_{2}\text{ on } (0, L)\times (0, T).
\end{equation}
We complement the fluid and structure boundary conditions with
\begin{equation}\label{fluid.boundary.condition}
\begin{aligned}
&\ug_{\gamma}=0\text{ on }(0,L)\times \{0\},\\
&\eta_{\gamma}\text{ and }\ug_{\gamma}\text{ are }L\text{-periodic with respect to }x.
\end{aligned}
\end{equation}
Note that  the kinematic condition \eqref{interface.condition} together with the incompressibility constraint of the fluid velocity   imply that, by taking into account the boundary conditions \eqref{fluid.boundary.condition},
\begin{equation}\label{compatibilite-vitesse-structure}
\int_0^L\partial_t \eta_\gamma(t,x) {\rm d}x =0.
\end{equation}
This condition states that the volume of the fluid cavity is preserved.
This condition implies that the pressure $p_{\gamma}$ is uniquely determined in contrast with classical fluid--solid interaction problems.
Finally  the fluid--structure system is completed with the following initial conditions
\begin{equation}\label{initial.conditions}
\begin{aligned}
&\eta_{\gamma}(0)=\eta^{0}_\gamma\text{ and }\partial_{t}\eta_{\gamma}(0)=\eta_\gamma^{1}\text{ in }(0,L),\\
&\ug_{\gamma}(0)=\ug_\gamma^{0}\text{ in }\mathcal{F}_{h_\gamma^{0}}\text{ with }h_\gamma^{0}=1+\eta_\gamma^{0}.
\end{aligned}
\end{equation}
\begin{remark}\label{rem:korn}
As already noted in \cite{Chambolle-etal}, due to the incompressibility constraint and the only transverse displacement of the beam
$$\left((\nabla \ug_\gamma)^T\cdot\normal_\gamma\right)_2=0, \text{ on }\Gamma_{h_{\gamma}(t)}.$$
It implies that the force applied by the fluid on the beam can be defined as follows
$$
\int_{0}^{L}\phi(\ug_{\gamma},p_{\gamma},\eta_{\gamma})\cdot \varphi(x,h_{\gamma}(x))\BS{e}_{2}\dd x=\int_{\Gamma_{h_{\gamma}(t)}}(\nabla\ug_{\gamma}-p_{\gamma}I_2)\normal_{\gamma}\cdot \varphi_{\vert\Gamma_{h_{\gamma}(t)}}\BS{e}_{2}\dd\Gamma_{h_{\gamma}(t)}.$$
For the same reason, a Korn equality also holds true
$$
\int_{\mathcal{F}_{h_\gamma(t)}} \vert\nabla\ug_{\gamma} +( \nabla\ug_{\gamma})^{T}\vert^2 =2\int_{\mathcal{F}_{h_\gamma(t)}} \vert\nabla\ug_{\gamma} \vert^2.
$$
\end{remark}

The fluid--structure system \eqref{Navier-Stokes}--\eqref{initial.conditions} is denoted by $(FS)_{\gamma}$ and $(FS)_0$ corresponds to the system with $\gamma=0$ for which we are going to prove the existence of a global weak solution. 
The case where $\gamma>0$ is the one considered in \cite{Grandmont-Hillairet}. It is proven therein that the structure does not touch the bottom of the fluid cavity, namely $\min_{x\in (0, L)}(1+\eta_\gamma(x, t)) >0$, for all $t$, implying the  existence of a unique global strong solution. In the case $\gamma=0,$ it is proven in \cite{Grandmont08, Muha-Canic13}  that a weak solution exists as long as the structure does  not touch the bottom of the fluid cavity. In this paper, we  investigate the vanishing viscosity limit ({\em i.e.} $\gamma\rightarrow 0$) and prove the convergence, up to the extraction of a subsequence, of the sequence of strong solutions $(\ug_\gamma,\eta_\gamma)$ solutions of $(FS)_{\gamma}$ defined on any time interval $(0, T)$ towards $(\ug, \eta)$ a weak solution (to be properly defined later on) of $(FS)_0$. Note that  at the limit  we loose the no contact property and have only:
 $\min_{x\in (0, L)}(1+\eta(x, t))\geq 0$, for all $t$. One key issue is thus to define an appropriate framework in case of contact. Moreover, and  as it is standard for this kind of fluid--structure coupled problem, another important difficulty comes from the obtention of  strong compactness of the approximate velocities. Such a property is mandatory in order to pass to the limit in the convective terms. 
 
 \medskip
 
 To conclude this introductory part, we point out that we do not address here the uniqueness of solutions. One reason is the lack of contact dynamics that should be added in case of contact, Hence, it is likely that our definition of weak solution below allows various rebounds of the elastic structure in case of contact (and consequently various solutions),  one (or several) of these solutions coming from the construction process we consider herein. This issue is now well--identified in the fluid--solid framework \cite{Starovoitov03}.

\medskip
 
The rest of the paper splits into two sections. In the next section, we introduce and analyze a functional setting, we give the definition of weak solutions and state the main result.
 The last section is devoted to the proof of the existence result following  standard steps: construction of a sequence of approximate solutions, derivation of compactness properties, passage to the limit. In the appendix, detailed proofs of technical lemma are given.

\section{Problem setting}
In this section we first recall the energy estimates satisfied by any regular enough solution of the coupled problem. We  then  construct functional spaces  and  introduce a notion of weak solution relying on these energy estimates  and  compatible with a contact. Finally, we provide the rigorous statement of our main result and some technical lemma necessary to the following analysis.

\subsection{Energy estimates}
Let $\gamma \geq 0 $ and assume that  $(\ug_\gamma, \eta_\gamma)$ is a classical solution to $(FS)_{\gamma}.$ Let then multiply the first equation of Navier--Stokes system \eqref{Navier-Stokes} by the fluid velocity $\ug_\gamma$ and integrate over $\mathcal{F}_{h_\gamma(t)}$. Let also multiply the beam equation \eqref{Beam.equation} with the structure velocity $\partial_t\eta_{\gamma}$ and integrate over $(0, L)$. By adding these two contributions, after integration by parts in space -- and by taking into account the coupling conditions (definition  \eqref{def:forcefluide} of $\phi$ and the kinematic condition \eqref{interface.condition}), the boundary conditions \eqref{fluid.boundary.condition} together with the incompressibility constraint and Remark \ref{rem:korn} -- we obtain
\begin{equation}\label{energy.equality}
\begin{aligned}
&\frac{1}{2}\frac{d}{dt}\left(\rho_{f}\int_{\mathcal{F}_{h_{\gamma}(t)}}\vert\ug_{\gamma}(t, \xg) \vert^{2} \dd\xg+ \rho_{s}\int_{0}^{L}\vert\partial_{t}\eta_{\gamma}(t, x) \vert^{2}\dd x+\beta\int_{0}^{L}\vert\partial_{x}\eta_{\gamma}(t, x) \vert^{2}\dd x+\alpha\int_{0}^{L}\vert\partial_{xx}\eta_{\gamma}(t, x)\vert^{2} \dd x\right)\\
&+\gamma\int_{0}^{L}\vert\partial_{tx}\eta_{\gamma}(t, x)\vert^{2}  \dd x+ \mu\int_{\mathcal{F}_{h_{\gamma}(t)}}\vert \nabla\ug_{\gamma}(t, \xg) \vert^{2}\dd \xg=0.
\end{aligned}
\end{equation} 
Note that we have used here  that the set ${\mathcal{F}_{h_{\gamma}(t)}}$ moves with the velocity field $\ug_{\gamma}$ thanks to the equality of velocities at the interface \eqref{interface.condition}, that implies
$$\int_{\mathcal{F}_{h_{\gamma}(t)}}(\partial_t\ug_\gamma(t, \xg)+(\ug_{\gamma}(t, \xg)\cdot\nabla )\ug_{\gamma}(t, \xg))\cdot \ug_{\gamma}(t, \xg)\dd\xg = \frac{1}{2}\frac{d}{dt}\int_{\mathcal{F}_{h_{\gamma}(t)}}\vert\ug_{\gamma}(t, \xg) \vert^{2} \dd\xg .$$
For $t >0,$ integrating \eqref{energy.equality} over $(0, t)$ leads to
\begin{equation}\label{energy.estimates}
\begin{aligned}
&\frac{1}{2}\left(\rho_{f}\int_{\mathcal{F}_{h_{\gamma}(t)}}\vert\ug_{\gamma}(t, \xg) \vert^{2} \dd\xg + \rho_{s}\int_{0}^{L}\vert\partial_{t}\eta_{\gamma}(t, x) \vert^{2} \dd x+\beta\int_{0}^{L}\vert\partial_{x}\eta_{\gamma}(t, x) \vert^{2} \dd x+\alpha\int_{0}^{L}\vert\partial_{xx}\eta_{\gamma}(t, x) \vert^{2} \dd x\right.\\
&\left.+\gamma\int_0^t\int_{0}^{L}\vert\partial_{tx}\eta_{\gamma}(s, x) \vert^{2} \dd x \dd s+ \mu\int_0^t\int_{\mathcal{F}_{h_{\gamma}(s)}}\vert \nabla\ug_{\gamma}(s, \xg) \vert^{2} \dd\xg\dd s \right)=\\
 &\frac{1}{2}\left(\rho_{f}\int_{\mathcal{F}_{h^0_{\gamma}}}\vert\ug^0_{\gamma}\vert^{2} \dd \xg+ \rho_{s}\int_{0}^{L}\vert\eta^1_{\gamma}\vert^{2}\dd x+\beta\int_{0}^{L}\vert\partial_{x}\eta^0_{\gamma}\vert^{2}\dd x+\alpha\int_{0}^{L}\vert\partial_{xx}\eta^0_{\gamma}\vert^{2}\dd x\right).
\end{aligned}
\end{equation}

As a consequence, we observe that, if $(\ug^0_\gamma, \eta^1_\gamma, \eta^0_\gamma)$ are such that the right-hand side of \eqref{energy.estimates} is uniformly bounded with respect to the viscosity parameter $\gamma \geq 0$, we have in particular 
\[
\eta_{\gamma}\text{ is uniformly bounded in }L^{\infty}(0,T;H^{2}_{\sharp}(0,L))\cap W^{1,\infty}(0,T;L_{\sharp}^{2}(0,L)),
\]
where the subscript $\sharp$ denotes spaces of periodic functions with respect to $x$.
Thus the associated interface displacements $(\eta_\gamma)_{\gamma \geq 0}$ are uniformly bounded at least in $\mathcal{C}^{0}([0,T];\mathcal{C}^{1}_{\sharp}(0,L))$ thanks to the compact embedding
\begin{equation}
L^{\infty}(0,T;H^{2}_{\sharp}(0,L))\cap W^{1,\infty}(0,T;L_{\sharp}^{2}(0,L))\hookrightarrow  \mathcal{C}^{0,1-s}([0,T];\mathcal{C}^{1,2s-\frac{3}{2}}_{\sharp}(0,L)), \qquad \forall \, \tfrac{3}{4}<s<1.
\label{inclusion-eta}
\end{equation}
Then,  there exists  $M>0$ depending on the initial data and independent of $\gamma$ such that
\begin{equation}
\label{borne.eta}
0\leq 1+\eta_\gamma (t, x) \leq M , \quad \forall (x, t) \in [0, L]\times [0, T], \quad  \forall \, \gamma  \geq 0.
\end{equation}

Finally, to define our functional setting, we rely below on the assumption  that the initial data $(\ug^0, \eta^1, \eta^0)$ associated to $(FS)_0$ do satisfy the assumption that the right-hand side of \eqref{energy.estimates} is finite (for $\gamma =0$).  So that we have at-hand an upper bound $M>0$ for the structure deformation $h=1+\eta$ for any physically reasonable solution.
The above computations show also that, up to a good choice of regularized initial data 
the same functional framework can be used to describe the solutions to the damped system $(FS)_{\gamma}$ (for $\gamma >0$).

\subsection{Functional spaces}
We design now a functional framework compatible with possible contact between the structure and the bottom of the fluid cavity. The parameter $M > 0$ is fixed in the whole construction.

\medskip

Given a non-negative function $h \in \mathcal{C}^{1}_{\sharp}(0,L)$  such that $0\leq h  \leq M$ we recall that we denote:
\[
\mathcal{F}_{h}=\{(x,y)\in\mathbb{R}^{2}\mid 0<x<L,\,0<y<h(x)\}.
\]

In case $h$ vanishes two crucial difficulties appear. 
First, the set $\mathcal{F}_{h}$ does not remain  connected (see Figure \ref{fig_cusp}, the domain below the graph splits into a connected component between the red dots and a connected component outside the red dots). In particular, if $h$ is the deformation of a structure associated with a solution $(\ug,p,\eta)$ to
$(FS)_0,$ we may expect that  the condition \eqref{compatibilite-vitesse-structure} must be satisfied on each  time--dependent connected component of the subset $\{x  \in (0, L) \text{ s.t. } h(x) >0\}$ and not only globally on $(0,L).$
Secondly, the boundaries of $\mathcal{F}_{h}$ contain at least one ``cusp'' so that it does not satisfy the cone property (see \cite{Adams}). As a consequence, one must be careful in order to define a trace  operator on $H^{1}(\mathcal F_h)$.

\begin{figure}[!ht]
\centering
\begin{tikzpicture}[scale=0.65]
\draw (0,0)node [below left] {$0$};
\draw (10,0)node [below right] {$L$};
\draw[red] (2.5,0) node {$\bullet$};
\draw[red] (8,0) node {$\bullet$};
\draw (0,0) -- (10,0);
\draw[dashed] (10,0) -- (10,5);
\draw[dashed] (0,5) -- (0,0);
\draw[red] (5.5,0.25)node [above] {$h(x)$};
\draw[red, very thick] (0,5) .. controls (1,5) and (1.5,0.1) .. (2.5,0);
\draw[red, very thick] (2.5,0) .. controls (5,0) and (5,2)  .. (6,1);
\draw[red, very thick] (6,1) .. controls (6.5,0.5) and (6.5,0.125)  .. (8,0);
\draw[red, very thick] (8,0) .. controls (10.5,0.25) and (8.5,5) .. (10,5);
\end{tikzpicture}
\caption{Example of a set with two cusps'. \label{fig_cusp}}
\end{figure}
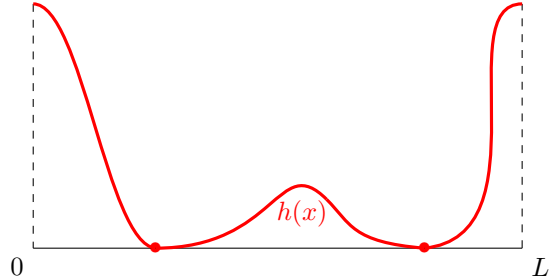

To overcome the second difficulty, we adapt the construction done in the context of fluid--solid problems in \cite{SanMartin-Starovoitov-Tucsnak}. Namely, we extend the fluid velocity fields -- by taking into account their trace on the structure --
on a time--independent domain whose regularity does not suffer from possible contacts.

\medskip

First, let us make precise some specific notations for the various domains used in the analysis. We introduce a virtual container $\Omega = (0,L) \times (-1,2M).$ This set contains a part of the substrate ($(x,y) \in (0,L) \times (-1,0)$), the fluid film
($(x,y) \in \mathcal F_h$)  and a virtual medium containing an extension of the structure (what remains of $\Omega$).  
Correspondingly, we also introduce three kinds of subsets of $\Omega.$ Given a continuous positive function $h$ we define first a subgraph domain (containing the substrate and the fluid film)
\[\mathcal{F}^-_{h}=\{(x,y)\in\Real^{2}\mid 0<x<L,\,-1<y<h(x)\},\]
then the  epigraph domain (corresponding to the virtual elastic medium)
\[\mathcal{S}_{h}=\{(x,y)\in\Real^{2}\mid 0<x<L,\,h(x)<y<2M\}.\]
Finally, for the analysis, we need also more general sets. Given $a,b:(0,L)\rightarrow \Real$ such that $a \leq b$,
we also define the set
\[\mathcal{C}_{a}^{b}=\{(x,y)\in\Real^{2}\mid 0<x<L,\,a(x)<y<b(x)\}.\]
We emphasize that there is some overlap between these notations. 
In particular, $\CM,  \mathcal{F}_{h}, \mathcal{F}^-_{h}, \mathcal{S}_{h}$ can be seen as particular cases of sets of the form $\mathcal{C}_{a}^{b}$.

For the study of non cylindrical time--dependent problems, we also need notations for space--time domains. 
We use the convention that notations for time--independent domains extend to the  time--dependent case by 
adding a hat. More precisely, we denote $\CMT=\CM\times(0,T)$ and 
\[\begin{aligned}
&\widehat{\mathcal{F}}_{h}=\bigcup_{t\in(0, T)}\mathcal{F}_{h(t)}\times\{t\}, &\widehat{\mathcal{F}}_{h}^{-}=\bigcup_{t\in(0, T)}\mathcal{F}_{h(t)}^{-}\times\{t\},\\
&\widehat{\mathcal{S}}_{h}=\bigcup_{t\in(0, T)}\mathcal{S}_{h(t)}\times\{t\}, &\widehat{\mathcal{C}}_{a}^{b}=\bigcup_{t\in(0, T)}\mathcal{C}_{a(t)}^{b(t)}\times\{t\},
\end{aligned}\]
where $h,a,b:(0,L)\times(0,T)\rightarrow \Real$ are such that $h(x,t)\geq 0$ and $a(x,t)\leq b(x,t)$ for all $(x,t)\in (0,L)\times(0,T)$. 
We will also denote by $\BS{n}_h$ the vector 
\[
\normal_{h}=\frac{1}{\sqrt{1+(\partial_{x}h)^{2}}}\begin{pmatrix}
-\partial_{x}h\\
1
\end{pmatrix}.
\] 
With these notations for the different sets, we introduce functional spaces
to which our  weak solutions will belong. 
The definition of these spaces is based on the following construction. 
Let  us first  introduce an extension operator:
\begin{definition}
\label{def:bar} 
Assume that $h\in \mathcal{C}^0_\sharp(0, L)$ with $0 \leq h \leq M.$ 
Let $\vg \in \BS{L}^2_{\sharp}(\mathcal{F}_h)$ and  $d\in L^2_\sharp(0, L)$, we define the extension operator by
$$\overline \vg= \left\{
\begin{array}{ll} 
d \BS{e}_{2}, &\text{ in } \mathcal{S}_h,\\
\vg,&\text{ in } \mathcal{F}_h,\\
\BS{0},& \text{ in } \mathcal{C}_{-1}^0.
\end{array}
\right.
$$
\end{definition}
\begin{remark}
By construction, this extension operator defines a vector field  $\overline \vg  \in \BS{L}^2_\sharp(\CM).$
In the previous definition the used symbol $\overline{\vg}$ involves only $\vg$ while the construction depends also on  $d$. In what follows, this choice is justified as we consider functions $\vg$ and $d$ satisfying the relation $\vg_{|_{y=h}}= d\BS{e}_{2}$, where $\vg_{|_{y=h}}$ denotes the function $x\mapsto \vg(x, h(x))$ on $(0, L)$.
\end{remark}
More precisely, when there is no contact this extension operator enjoys the following properties:
\begin{lemma}
\label{lem:bar}
Assume that $h\in W^{1, \infty}_\sharp(0, L)$ with $0<h(x)\leq M$ for $x\in [0, L]$ and let $s \in (0,1).$
\begin{enumerate}
\item  If $s > 1/2$  and $\vg \in \BS{H}^s_{\sharp}(\mathcal{F}_h)$ is divergence free with $\vg_{|_{y=0}}=0$, and $\vg_{|_{y=h}}=d \BS{e}_{2}$ with $d\in H^{s}_\sharp(0, L),$ we have that
\begin{align*}
& \overline{\vg} \in \BS{H}^s(\Omega)\,, \qquad  {\rm div }\, \overline{\vg}  = 0 \text{ in $\Omega$},  \qquad \overline\vg\cdot\BS{e}_1=0 \text{ in $\mathcal{S}_h$}.
\end{align*} 

\item If $0\leq s<1/2$ and $\vg \in \BS{H}^s_{\sharp}(\mathcal{F}_h)$ is divergence free with $\vg\cdot\BS{e}_2=0$ on $y=0$, and $\vg_{|_{y=h}} \cdot\BS{n}_h=(0, d)^T \cdot\BS{n}_h$ on $(0,L)$  with $d\in H^s_\sharp(0, L),$  we have that:
\begin{align*}
& \overline{\vg} \in \BS{H}^s(\Omega)\,, \qquad  {\rm div }\, \overline{\vg}  = 0 \text{ in $\Omega$},  \qquad \overline\vg\cdot\BS{e}_1=0 \text{ in $\mathcal{S}_h$}.
\end{align*} 
\item In both cases $0\leq s < 1/2$ and $s>1/2$ the extension operator is a bounded linear mapping of its arguments
whose norm can be bounded w.r.t $M$ only: 
\[
\|\overline{\vg}\|_{\BS{H}^s(\Omega)} \leq C(M)\left(  \|d\|_{H^s_{\sharp}(0,L)} + \|\vg\|_{H^s(\mathcal F_h)} \right).
\]
\end{enumerate}
\end{lemma}

\begin{proof}
We note that in both cases $\bar{\vg}$ is by construction piecewisely divergence free and belongs to $\BS{H}^s$ (in the sets  $ \mathcal{S}_h$, $ \mathcal{F}_h,$ $ \mathcal{C}_{-1}^0$). Consequently, in case (2) the extension is straightforwardly in $\BS{H}^s(\Omega).$ Only the continuity of normal traces is required to yield a global divergence free vector field. In Case (1) we require continuity of the full trace to obtain an $\BS{H}^s(\Omega)$ 
vector field.
\end{proof}

\begin{remark} 
 \label{rem:domaine}
 \begin{enumerate}[(i)]
 \item In the case $\min_{x\in [0, L]} h(x)\geq 0$,  we may extend vector fields defined on $\mathcal{F}^-_h$ with a similar bar-operator. Then, similar results for this extension operator hold true.
 
\item In  Lemma \ref{lem:bar} and in what follows, in order to avoid to denote the trace by the classical symbol $\gamma$, which is reserved here to the added viscosity on the structure, we denote by ${\boldsymbol v}_{|_{y=h}}$ the trace  of $\vg$ defined as ${\boldsymbol v}_{|_{y=h}}(x)=\vg(x, h(x))$. We note that when $h(x)>0$ for all $x\in [0, L]$, the associated linear trace operator is well defined from $H_\sharp^1(\mathcal{F}_h)$  into $H_\sharp^{\frac{1}{2}}(0, L)$. In the case where $h(x)\geq0$ for all $x\in [0, L]$ it is well defined from $H_\sharp^1(\mathcal{F}^-_h)$ into $H_\sharp^{\frac{1}{2}}(0, L)$.  It is easy to verify that 
\begin{equation}
\label{est:trace}
\|\vg_{|_{y=h}}\|_{H^\frac{1}{2}_\sharp(0, L)}\leq C(\|h\|_{W_\sharp^{1, \infty}(0, L)}) \|\vg\|_{H^1_\sharp(\mathcal{F}^-_h)}.
\end{equation}

\end{enumerate} 
\end{remark}
Consequently, for a $W^{1, \infty}_{\sharp}(0,L)$--function $h$ satisfying $0 \leq h(x) \leq M$, for $x\in [0, L]$ and for $s \in (0,1),$ we set
\begin{align}
&K^{s}[h]=\{\vg\in \BS{H}^{s}_{\sharp}(\CM)\mid \DIV{\vg}=0\text{ in }\CM,\,\vg=0\text{ in }\mathcal{C}_{-1}^{0},\,\,\vg\cdot\BS{e}_{1}=0\text{ in }\mathcal{S}_{h}\},\label{espace.Ks}\\
&X^{s}[h]=\{(\wg,d)\in K^{s}[h]\times (H^{2s}_{\sharp}(0,L)\cap L_{\sharp, 0}^{2}(0,L))\mid {w_{2}}_{|_{(0, L)\times \{M\}}}=d\},\label{espace.Xs}
\end{align}
where 
$$L_{\sharp, 0}^{2}(0,L)= \left\{d\in L_{\sharp}^{2}(0,L) \, s. t. \, \int_0^L d =0\right\}.$$
When $s=0$ we denote $K[h] = K^0[h]$ and $X[h] = X^0[h].$

Under the assumptions of Lemma \ref{lem:bar} 
we have that $\bar{\vg} \in K^s[h]$ and $(\bar{\vg},d) \in X^s[h]$ in both cases $s \in (0,1/2)$ and $s \in (1/2,1)$.
We emphasize that,  for any $\vg\in K[h]$, the divergence free condition implies that the trace on ${(0, L)\times \{M\}}$ of ${v_{2}}=\vg\cdot\BS{n}$ has a sense in $H^{-1/2}_\sharp(0, L)$. Similarly $\vg_{|_{y=h}}\cdot \BS{n}_h$   also makes sense in $H^{-1/2}_\sharp(0, L)$. Following the construction of the extension operator above, one expects this trace to represent the structure velocity. 
 \medskip 
 Correspondingly, we introduce smooth variants of these functional spaces $\mathcal{K}[h]$ and  $\mathcal{X}[h]$ defined by
\begin{align}
&\mathcal{K}[h]=\{\wg\in\mathcal{C}^{\infty}_{\sharp}(\CM)\mid \DIV{\wg}=0\text{ in }\CM,\,\wg=0\text{ in }\mathcal{V}(\mathcal{C}_{-1}^{0}),\,\wg\cdot\BS{e}_{1}=0\text{ in } \mathcal V(\mathcal{S}_{h})\},\\
&\mathcal{X}[h]=\{(\wg,d)\in \mathcal{K}[h]\times  \mathcal{C}^{\infty}_{\sharp}(0,L)\mid {w_{2}}_{|_{(0, L)\times \{M\}}}=d\}.
\end{align}
Here, we used``in $\mathcal{V}(\mathcal O)$'' as a shortcut for the  statement ``in a neighbourhood of the open set $\mathcal O$".

\medskip

Before defining  the weak solutions, we now verify  that the previous coupled spaces  encode
the fluid--structure nature of the problem and behave correctly (from an analytical standpoint). 
Once again, $h$ stands for a non--negative $W^{1, \infty}$--function satisfying $0\leq h \leq M$. 
The space $X[h]$ is endowed with the scalar product
\begin{equation}
\label{def:produitscalaire}
\langle (\ug,\overset{\cdot}{\eta}),(\wg,d)\rangle_{X[h]}:=\rho_{f}\int_{\CM}\ug\cdot\wg + \rho_{s}\int_{0}^{L}\overset{\cdot}{\eta}d,\end{equation}
and  we endow the spaces $X^s[h]$ with a Hilbert structure associated with the norms
\[
\|(\wg,d)\|_{X^s[h]}=  \|\wg\|_{H^s(\Omega)} +  \|d\|_{H^{2s}(0,L)}.
\]
For $s=0$ this Hilbert-norm does not correspond to the scalar product as defined in \eqref{def:produitscalaire} 
but the topologies are equivalent since $\rho_f$ and $\rho_s$ are both positive.

In order to prove the fluid--structure property, we show in the following lemma that,
in the ``virtual medium", the velocity--fields in $X[h]$ coincide with a structure velocity.
\begin{lemma}Let $\vg  \in K[h].$ There exists $d\in L_{\sharp}^{2}(0,L)$ such that $\vg=d\BS{e}_{2}$ in $\mathcal{S}_{h}$. \end{lemma}
\begin{proof}
By definition we have $\vg=(0,v_{2})^{\top}$ in $\mathcal{S}_{h}$. Moreover the divergence condition
$\DIV{\vg}=\partial_yv_{2}=0\text{ in }\mathcal{S}_{h},$
implies that $v_{2}(x,y)=v_{2}(x)$ in $\mathcal{S}_{h}$. Since $\vg\in \BS{L}^2_\sharp(\CM)$ and $0\leq h(x) \leq M, \forall x\in [0, L]$, we have $v_2\in L_\sharp^2(0, L)$. 
\end{proof}

\medskip

Given a divergence free $\wg \in {L}^2_{\sharp}(\Omega)$
it is classical that we can construct a stream function $\Psi \in {H}^1_{\sharp}(\Omega)$  such that $\wg = \nabla^{\bot} \Psi.$ We show in the following lemma some additional properties satisfied by the stream function of an extended-field in $K[h]$:

\begin{lemma}\label{lemma.X[h]}
Let $(\wg, d)\in X[h]$ and set $I=\{x\in[0,L]\mid h(x)>0\}$. There exists $\Psi \in {H}^1_\sharp(\Omega)$ such that $\wg=\nabla^\perp\Psi$ which furthermore satisfies
\begin{itemize}
\item $\Psi(x,y)=b(x)$ in $\mathcal{S}_h$ with $b\in H^{1}_{\sharp}(0,L)$;
\item $\Psi =0$ in $\mathcal{C}_{-1}^0$;
\item $\Psi=0$ in $I^{c}\times(-1,2M)$.
\end{itemize}
\end{lemma}
\begin{proof}
We note that $\Psi$ is defined up to an additive constant.
However, in $\mathcal{C}_{-1}^0$ we have $\wg=0$, so that we fix this constant by choosing $\Psi=0$ in $\mathcal{C}_{-1}^0$. Then, 
due to the previous lemma  $\wg_{\vert_{(0, L)\times \{M\}}}= d\BS{e}_{2}$ and the identity $\wg=d\BS{e}_{2}$ holds in $\mathcal{S}_{h}$. Thus $\partial_{x}\Psi=d$ in $\mathcal{S}_{h}$ and $\Psi(x,y)=b(x)$ in $\mathcal{S}_{h}$ where $b\in H^{1}_{\sharp}(0,L)$ satisfies $\partial_x b = d$. Remark that the $L$--periodicity of $b$ is ensured by $\int_{0}^{L}d(s)\dd s =0$. 

\medskip

Concerning the last point of the lemma, we emphasize that,
since $\Psi \in \BS{H}^1_{\sharp}(\Omega)$, its trace is well defined
on vertical lines $x=cst.$ Consequently, the value of $\Psi$
on $I^{c} \times (-1,2M)$ is well defined, whatever the topological properties of $I^c$ are. Now, given $a \in I^{c}$ (assuming $I^{c}$ is non-empty), we have $h(a)=0$ (by definition of $I^c$). The identity $\Psi(x,y)=b(x)$ in $ \mathcal{S}_{h}$, with  $b\in H^{1}_{\sharp}(0,L)$, implies that $\Psi\in\mathcal{C}_\sharp^{0}(\mathcal{S}_{h})$. In particular, $\Psi$ is equal to a constant $b(a)$ on $\{a\}\times(-1,2M)$. Moreover, the function $\Psi$ is equal to $0$ on $\mathcal{C}_{-1}^{0}$. 
Finally, applying by a trace argument that  $\Psi \in H^{1/2}(\{a\}\times (-1,2M))$ and using for example the following definition of the $H^{1/2}$-norm
\[\norme{\Psi}_{H^{1/2}(\{a\}\times(-1,2M))}^{2}=\int_{\{a\}\times(-1,2M)}\Psi^{2} + \int_{\{a\}\times(-1,2M)}\int_{\{a\}\times(-1,2M)}\frac{\vert\Psi(a,x)-\Psi(a,y)\vert^{2}}{\vert x-y\vert^{2}},\]
we get that the trace of  $\Psi$ cannot ``jump'' in $y=0.$ 
Therefore, we have  $b(a)=0$ and $\Psi=0$ on $\{a\}\times(-1,2M)$. 
\end{proof}

\medskip

We conclude this preliminary analysis of the space $X[h]$
by showing that we have density of smooth vector fields in 
$X[h]$. This is made precise in the following lemma:

\begin{lemma}\label{lemma.density}The embedding $X[h]\cap (\mathcal{C}_\sharp^{\infty}(\overline{\CM})\times\mathcal{C}_\sharp^{\infty}(0, L))\subset X[h]$ is dense.
\end{lemma}
\begin{proof}
 The difficulty of this proof is to deal with the case where $h$ has zeros. The main idea  is to work with the stream function 
 of the extended vector field. If we had $h(x)>0$, for all $x\in [0, L]$, a contraction in $y$ and a standard truncature and regularization argument on the stream function can be used. In the case where $h$ has zeros, one first cuts off the zeros of $h$ and then takes advantage of  the better regularity of $\Psi$ on the structure. A detailed proof is given in the Appendix \ref{Annexe}.
\end{proof}


\subsection{Weak solutions and main result.}

In this section we introduce first our weak formulation of $(FS)_{\gamma}$. 

\medskip

 We assume that the initial conditions $(\ug^{0},\eta^{0},\eta^{1})$ satisfy
\begin{align}
&\eta^{0}\in H^{2}_{\sharp}(0,L) \text{ with }\displaystyle\min_{x\in [0,L]}(1+\eta^{0})>0,\label{CI1}\\
&(\ug^0,\eta^{1})\in \BS{L}_\sharp^{2}(\mathcal{F}_{h^{0}})\times L_{\sharp, 0}^{2}(0,L),\label{CI2}\\
& \DIV{\ug^{0}}=0 \text{ in }\mathcal{F}_{h^{0}},\label{CI3}\\
&\ug^{0}\cdot \BS{n}^0 = 0 \text{ on } (0, L)\times \{0\} \text{ and } \ug^{0}(\cdot, h_0(\cdot))\cdot \BS{n}^0= (0, \eta^{1}(\cdot) )^T\cdot \BS{n}^0\text{ on }(0, L).\label{CI4}
\end{align}
We can then define $M>0$ by \eqref{borne.eta} and construct the associated $\CM.$  
We have  the following definition for a weak solution to $(FS)_{\gamma}$:

\begin{definition} \label{def.ws}
Let $(\ug^{0},\eta^{0},\eta^{1})$ satisfying \eqref{CI1}--\eqref{CI3} and $\gamma >0.$
We say that a pair $({\overline{\ug}_\gamma},\eta_\gamma)$ is a weak solution to $(FS)_\gamma$ if it satisfies the following items:
\begin{enumerate}[i)]
\item $({\overline{\ug}_\gamma},\eta_\gamma)\in L^{\infty}(0,T;\BS{L}_{\sharp}^{2}(\CM))\times \left(L^{\infty}(0,T;H^{2}_{\sharp,  0}(0,L))\cap W^{1,\infty}(0,T;L_\sharp^{2}(0,L))\right)$ with 
\[
({\overline{\ug}_\gamma}(t), \partial_t\eta_\gamma(t))\in X[h_\gamma(t)] \text{ for a.e. $t\in(0,T),$} \quad
 \nabla \overline{\ug}_\gamma \in  L^{2}(\hat{\mathcal{F}}^{-}_{h_\gamma}), 
\]
\item   the kinematic condition
$$
\ug_\gamma(t, x, 1+\eta_\gamma(t, x))=\partial_t\eta_\gamma(t,x) {\bf{e}}_2 \quad \textrm{ on } (0, T) \times (0, L),
$$
\item For any $(\wg_\gamma, d_\gamma)\in \mathcal{C}_\sharp^{\infty}\left(\overline{\CMT}\right)\times\mathcal{C}_\sharp^{\infty}( [0, L]\times [0, T])$ such that $(\wg_\gamma(t), d_\gamma(t))\in\mathcal{X}[h_\gamma(t)]$ for all $t\in [0,T]$ we have for a.e. $t\in(0,T)$
\begin{equation}\label{weak.formulation.FS}
\begin{aligned}
&\rho_{f}\int_{\mathcal{F}_{h_\gamma(t)}}\ug_\gamma(t)\cdot\wg_\gamma(t) -\rho_{f}\int_{0}^{t}\int_{\mathcal{F}_{h_\gamma(s)}}\ug_\gamma\cdot\partial_{t}\wg_\gamma + (\ug_\gamma\cdot\nabla)\wg_\gamma\cdot\ug_\gamma\\
&+\rho_{s}\int_{0}^{L}\partial_{t}\eta_\gamma(t)d_\gamma(t)-\rho_{s}\int_{0}^{t}\int_{0}^{L}\partial_{t}\eta_\gamma\partial_{t}d_\gamma+\mu\int_{0}^{t}\int_{\mathcal{F}_{h_\gamma(s)}}  \nabla\ug_\gamma:\nabla\wg_\gamma\\
&+\beta\int_{0}^{t}\int_{0}^{L}\partial_{x}\eta_\gamma\partial_{x}d_\gamma + \alpha\int_{0}^{t}\int_{0}^{L}\partial_{xx}\eta_\gamma\partial_{xx}d_\gamma+\gamma\int_{0}^{t}\int_{0}^{L}\partial_{xt}\eta_\gamma\partial_{x}d_\gamma \\&
=\rho_{f}\int_{\mathcal{F}_{h_0}}\ug^{0}\cdot\wg_\gamma(0)+\rho_{s}\int_{0}^{L}\eta^{1}d_\gamma(0).
\end{aligned}
\end{equation}
where $\ug_\gamma={\overline{\ug}_\gamma}_{\vert \widehat{\mathcal{F}}_{h_\gamma}}.$
\end{enumerate}
\end{definition}

The regularity statements in the first item of the definition comes from the energy estimate \eqref{energy.estimates} while
the weak formulation \eqref{weak.formulation.FS} is obtained classically by multiplying the fluid equation \eqref{Navier-Stokes}
with $\wg_{\gamma}$ and the beam equation \eqref{Beam.equation} with $d_{\gamma}$ and performing formal integration by parts. As  usual for this type of fluid--structure problem, the test functions  depend on the solution and thus on the parameter $\gamma,$ adding further nonlinearity to the system.

\medskip

We recall that, from \cite{Chambolle-etal}, \cite{Grandmont08},  there exists a weak solution for $\gamma \geq 0$ as long as the beam does not touch the bottom of the fluid cavity. If $\gamma >0$ again and the initial data are smooth enough, it is also  proved in \cite{Grandmont-Hillairet}  that there exists a unique global in time strong solution such that $\min_{x\in [0, L]}h_\gamma(x, t)>0$ for all $t>0$. 
The main result of this paper is stated in the following theorem

\begin{theorem}\label{main.theorem}
Suppose that $T>0$ and that the initial conditions $(\ug^{0},\eta^{0},\eta^{1})$ satisfy \eqref{CI1}-\eqref{CI4}.

Then $(FS)_0$ has a weak solution $(\ug,\eta)$ on $(0,T).$ This solution satisfies furthermore for a.e. $t \in (0,T)$
\begin{equation}\label{energy.estimates.thm}
\begin{aligned}
&\frac{1}{2}\left(\rho_{f}\int_{\mathcal{F}_{h(t)}}\vert\ug(t, \xg) \vert^{2} \dd\xg + \rho_{s}\int_{0}^{L}\vert\partial_{t}\eta(t, x) \vert^{2} \dd x+\beta\int_{0}^{L}\vert\partial_{x}\eta(t, x) \vert^{2} \dd x+\alpha\int_{0}^{L}\vert\partial_{xx}\eta(t, x) \vert^{2} \dd x\right)\\
& + \mu\int_0^t\int_{\mathcal{F}_{h(s)}}\vert \nabla\ug(s, \xg) \vert^{2} \dd\xg\dd s \leq \\
& 
.\qquad \qquad \frac{1}{2}\left(\rho_{f}\int_{\mathcal{F}_{h^0}}\vert\ug^0\vert^{2} \dd \xg+ \rho_{s}\int_{0}^{L}\vert\eta^1\vert^{2}\dd x+\beta\int_{0}^{L}\vert\partial_{x}\eta^0\vert^{2}\dd x+\alpha\int_{0}^{L}\vert\partial_{xx}\eta^0\vert^{2}\dd x\right).
\end{aligned} 
\end{equation}

\end{theorem}

Before detailing the proof of this result, we shall comment on the choice of test functions and the relations with a strong formulation of $(FS)_\gamma,$ in particular in the case where contacts occur. Since we focus on the construction of weak solutions, we stick here to a short description of formal arguments. Before contact ({\em i.e.} as long as $\min_{x \in [0,L]} h_\gamma(x,t) \geq \bar{\alpha}$ for some $\bar{\alpha} >0$), we claim that our definition coincides with the definition of \cite{Chambolle-etal} and that the solutions constructed in \cite{Grandmont-Hillairet} for smooth data match our definition also.  In particular in this case we can choose test functions such that $(\wg_\gamma(t), d_\gamma(t))\in  X^1[h_\gamma(t)]$. Before contact,
we recover \eqref{Navier-Stokes}--\eqref{Beam.equation} from the weak formulation using a classical argument. First, we may take as
test functions the  vector fields $\wg_\gamma \in \mathcal{C}^{\infty}_c(\hat \Omega)$ with $d = 0$ and we recover  \eqref{Navier-Stokes} with a zero mean pressure $p$ by adapting an argument of de Rham. Assuming that 
$(\ug_\gamma,p_\gamma)$ is sufficiently smooth -- to be able to define $\sigma(\ug_\gamma,p_\gamma) \BS{n}_{\gamma}$
on $\Gamma_{h_{\gamma}}$ -- we can also recover the beam equation \eqref{Beam.equation} with the following construction that enables to extend structure test functions in the fluid domain

\begin{definition}
\label{def:R} 
Let $\lambda >0$ and $\zeta \in \mathcal{C}^{\infty}(\mathbb R)$ such that 
$\mathbf{1}_{[1,\infty)} \leq \zeta \leq \mathbf{1}_{[1/2,\infty)}. $ 
Given $d\in L_{\sharp,0}^2(0, L)$, we define
$$
\mathcal{R}_\lambda(d)(x,y)= \nabla^{\bot} (b(x) \zeta(y/\lambda))   \quad
 \forall \, (x,y) \in \Omega.
$$
where $b \in H^{1}_{\sharp}(0,L) \cap L^2_{\sharp,0}(0,L)$ satisfies 
$\partial_x b = d.$
\end{definition}

We note that the above construction is well defined since $d$ is chosen to be mean free. We do not include the dependence on  $\zeta$ in the name of our operator since it will be a given  fixed function throughout the paper.  The present lifting operator is a variant of the one introduced in \cite{Chambolle-etal}.  It enjoys the following straightforward properties:

\begin{lemma}
Let $\lambda >0$ and $h\in W^{1, \infty}_\sharp(0, L)$ satisfying 
$\lambda\leq h(x)\leq M, \forall x\in [0, L]$.
\begin{enumerate}
\item  $\mathcal R_{\lambda}$ is a linear continuous mapping from $H^s_{\sharp}(0,L) \cap L^2_{\sharp}(0,L)$ into $K^s[h]$ for arbitrary $s \in [0,1].$
\item $\mathcal R_{\lambda}$ maps $\mathcal{C}^{\infty}_{\sharp}(0,L) \cap L^2_{\sharp,0}(0,L)$
into $\mathcal K[h].$
\end{enumerate}
\end{lemma}

Consequently, before contact, for any arbitrary structure test function $d_{\gamma} \in \mathcal{C}^{\infty}_{\sharp}(0,L) \cap L^2_{\sharp,0}(0,L)$ we may consider
\[
\wg_{\gamma} :=  \mathcal R_{{\lambda}}[d_{\gamma}] \in \mathcal{C}^{\infty}_{\sharp}(\overline{\hat{\Omega}}),
\] 
so that $(\wg_{\gamma},d_{\gamma})$ is an admissible test function in our weak formulation. Classical integration by parts argument then enables to recover the structure equation \eqref{Beam.equation} multiplied by $d_{\gamma}.$ We note that, in this way, we recover \eqref{Beam.equation} up to a constant (indeed the test function $d_{\gamma}$ is mean free), but this constant  mode corresponds to the choice of the constant normalizing the pressure in order to match the global volume preserving constraint 
\eqref{compatibilite-vitesse-structure}.

\medskip

When contact occurs, we recover a similar set of equations, assuming, once again, that the solution is sufficiently regular. Let consider for instance a simplified configuration such that, on some time interval $(T_0,T_1)$ there exist $\mathcal C^1$--functions $(a_k^{-},a_k^+): (T_0,T_1) \to \mathbb R^2$ ($k\in \mathbb N$) such that 
\[
\{(t,x) \in (T_0,T_1 )\times(0,L) \ | \ h(t,x) >0 \} = \bigcup_{k\in\mathbb N}  \bigcup_{t\in (T_0,T_1)}  \{t\} \times  (a_k^{-}(t),a_{k}^+(t))
\]
In that case, we can reproduce similar arguments developed in the no contact case to recover the Navier--Stokes equations
and the structure equations in each connected component of the fluid domain.
More precisely, let introduce
\begin{align*}
& \hat{\mathcal F}_k := \{(t,x,y) \in (T_0,T_1) \times (0,L) \times (0,M) \text{ s.t. } a_k^{-}(t) < x < a_k^{+}(t) \quad  0 < y  < h_{\gamma}(t,x) \}\,,
\\
& \hat{\Gamma}_k :=   \{(t,x) \in (T_0,T_1) \times (0,L)  \text{ s.t. } a_k^{-}(t) < x < a_k^{+}(t)   \}\,.
\end{align*}
First, by using that $\ug_{\gamma}$ is divergence free on $\hat{\mathcal F}_{k}$  we obtain 
\[
\int_{a_{k}^{-}(t)}^{a_k^{+}(t)} \partial_t \eta_{\gamma} = 0.
\]
Second, by taking as a fluid test function a velocity field $\wg_{\gamma}$ with compact support in $\hat{\mathcal F}_k$, 
we construct a pressure $p_{k,\gamma}$ on $\hat{\mathcal F_{k}}$  so that \eqref{Navier-Stokes} holds true.
We recall that, at this point, $p_{k,\gamma}$ is defined up to a constant. The global
pressure $p_{\gamma}$ is then constructed by concatenating all the $(p_{k,\gamma})_{k}$ to yield a pressure on $\hat{\mathcal F}_{h_{\gamma}}$
(that is defined up to a number of constants related to the number of parameters $k$).
Third, we consider a mean free structure test function $d_{\gamma} \in \mathcal{C}^{\infty}_c(\hat{\Gamma}_{k}).$ Since $d_{\gamma}$
has compact support in the open set where $h_{\gamma} >0 $,  $h_{\gamma}$ is bounded from below by some $\alpha_k>0$ on $\hat{\Gamma}_k$. So, instead of choosing the mean free anti-derivative $b_{\gamma}$ of the structure test function $d_{\gamma}$ in the definition of $\mathcal R_{{\alpha}_k}$, we choose the one that vanishes outside the support of $d_{\gamma}.$
In that way, we construct a test function $\wg_{\gamma}$ such that $(\wg_{\gamma},d_{\gamma})$ is  adapted to our weak formulation. So, we 
obtain \eqref{Beam.equation} on $\hat{\Gamma}_{k}$ up to a constant which is afterwards fixed by a suitable choice of the pressure
on $\hat{\mathcal F}_{k}$. {Note that the structure equation is recovered on each component  $\hat{\Gamma}_{k}$ and not on the whole interval $(0, L)$} and that the pressure is again uniquely defined but that there are more constants to fix than in the no contact case.
To end up this remark, we emphasize that when $\min_{x\in [0, L]}h_\gamma(x, t)>0, \forall t\in [0, T]$, the test functions can be chosen in $X[h_\gamma(t)]$. Moreover if $\min_{x\in [0, L]}h_\gamma(x, t)>0$ the elastic test functions can be chosen independent of the solution and thus independent of the regularization parameter $\gamma$ (see \cite{Chambolle-etal}). It is not the case when a contact occurs  since, as we saw in the previous construction, we require  $d_\gamma=0$ in a neighbourhood of the contact points.

\medskip

We end this part by giving a roadmap of our proof of Theorem \ref{main.theorem}.
To obtain a solution of the  variational problem  for $\gamma=0$ we consider the approximate fluid--structure system $(FS)_{\gamma}$ with a viscosity $\gamma>0$. From \cite[Theorem 1]{Grandmont-Hillairet-Lequeurre} this fluid--structure system $(FS)_{\gamma}$, completed with regularized  initial conditions $(\ug^{0}_{\gamma},\eta^{0}_{\gamma},\eta^{1}_{\gamma})$, admits a unique strong solution $(\ug_{\gamma},p_{\gamma},\eta_{\gamma})$ such that $\min_{x\in [0, L]}h_\gamma(x, t)>0$. It ensures that the existence time interval does not depend on $\gamma$. Moreover,  this solution satisfies the energy equality \eqref{energy.estimates}, and thus one can extract converging subsequences. 
We may then consider one cluster point of this sequence and show that this is a weak solution to $(FS)_0$.
 One key point here is that strong compactness of the approximate velocity fields is needed to pass to the limit in the convective nonlinear terms. The classical Aubin--Lions lemma does not apply directly because of the time-dependency of fluid domains and of the divergence free constraint.  Many different strategies may be used to handle this difficulty \cite{Fujita-Sauer, Grandmont08, Lengeler, Moussa}.  Here we follow the line of \cite{SanMartin-Starovoitov-Tucsnak} where the existence of weak solutions for a fluid--solid problem beyond contact is proven. We first obtain compactness of a projection of the fluid and structure velocities. Roughly speaking the idea is to obtain compactness on fixed domains  independent  of time and of $\gamma$. So, we  define an interface satisfying $0\leq\underline{h}\leq h_\gamma$, which is regular enough and ``close" to $h_\gamma$ for all $\gamma$ small enough 
 and we prove compactness of the projections of the velocity fields on the coupled space associated to $\underline{h}.$ We recover the 
 compactness of the velocity fields by proving some continuity properties of the projection operators with respect to $\underline{h}.$ {In particular we prove that $X^{s}[\underline{h}]$ is a good approximation space of $X^{s}[h]$ in $H^s$ for some $s\geq 0$ whenever $\underline{h}$ is close to $h$.}
 This part of the proof is purely related to the definition of the spaces $X[h]$. Consequently we detail the arguments as preliminaries in the next subsection. We emphasize again that, since one may loose the no contact property at the limit, this study on the compactness of approximate velocity fields requires specific constructions. Once compactness is obtained, we pass finally to the limit in the weak formulation.  Again, as  contact may occur in the limit problem, we cannot follow \cite{Chambolle-etal} to construct a dense family of test functions independent of $\gamma$ to pass to the limit.

\subsection{On the $h$-dependencies of the spaces $X^s[h]$}
\label{continuite_h}
In this section, we analyze the continuity properties of the sets $X^s[h]$ with respect to the parameter $h.$
To start with, we remark that, given $h \in \mathcal{C}^0_{\sharp}(0,L)$ satisfying $0 \leq h \leq M$ the space
$X^s[h]$ is a closed subspace of 
\[
X^s  :=  \BS{H}^s_{\sharp}(\Omega) \times H^{2s}_{\sharp}(0,L) .
\] 
We can then construct the projector $\mathbb P^s[h] : X^s \to X^s[h].$ We analyze in this section the continuity
properties of these projectors with respect to the function $h.$ Our main result is the following lemma:
 
\begin{lemma} \label{lem:projector}
Fix $0<\kappa<\frac{1}{2}$. Let $h$ and $\hb$ belong to $H_\sharp^{1+\kappa}(0,L)\cap W_\sharp^{1,\infty}(0,L)$ with $0\leq \hb\leq h\leq M$ and set
\begin{equation}
\label{borne:A}\norme{h}_{H^{1+\kappa}_{\sharp}(0,L)}+\norme{h}_{W^{1,\infty}_{\sharp}(0,L)}+\norme{\hb}_{H^{1+\kappa}_{\sharp}(0,L)}+\norme{\hb}_{W^{1,\infty}_{\sharp}(0,L)}\leq A.
\end{equation}
Let $s \in [0,\kappa/2)$ and $(\wg,\etap) \in X^s[h]$ enjoying the further property
\begin{eqnarray}
&\wg_{\vert\mathcal{F}_{h}^{-}}\in \BS{H}^{1}_{\sharp}(\mathcal{F}^{-}_{h}).\label{hyp-Ps2}
\end{eqnarray}
Then, the following estimate holds true:
\begin{equation}\label{est.projector.main}
\norme{\mathbb{P}^{s}[\hb](\wg,\etap) -(\wg,\etap)}_{X^s}\leq C_{A}(\norme{\hb -h }_{W^{1,\infty}_{\sharp}(0,L)})\norme{\wg}_{\BS{H}^{1}_{\sharp}(\mathcal{F}^{-}_{h})},
\end{equation}
where $C_{A}(x)\underset{x\rightarrow 0}{\longrightarrow}0$.
\end{lemma}

\begin{proof}
The idea is to construct $(\vg,b)\in X^{s}[\hb]$ such that 
\begin{equation}\label{est.projector}
\norme{(\vg,d) -(\wg,\etap)}_{\BS{H}^{s}_{\sharp}(\CM)\times H^{2s}_{\sharp}(0,L)}\leq C_{A}(\norme{\hb -h }_{W^{1,\infty}_{\sharp}(0,L)})\norme{\wg}_{\BS{H}^{1}_{\sharp}(\mathcal{F}^{-}_{h})}.
\end{equation}
The inequality \eqref{est.projector.main} then follows from the minimality property of the projection.
The proof is divided in five  steps. The two first ones are devoted to the construction of $(\vg,d).$ The latter ones
concern the derivation of \eqref{est.projector}.

\medskip

\noindent{\textbf{Step 1. Geometrical preliminaries.}} 
Before going to the construction of a candidate $(\vg,d)$ we define and analyze a change of variables 
$\chi$ that maps $\mathcal{F}^-_{\hb}$ on $\mathcal{F}^-_{h}$. For $(x,y) \in \Omega,$ we set
\[
\chi(x,y) = \left(x,m(x)(y+1)-1\right), \; \text{ where }  m(x) = \displaystyle\frac{h(x)+1}{\hb(x)+1}
\]
Clearly, $\chi$ realizes a one-to-one mapping between $\mathcal{F}^-_{\hb}$ and $\mathcal{F}^-_{h}$.
 Thanks to the regularity assumptions on $h$ and $\hb$, we remark that $m\in W_\sharp^{1,\infty}(0, L)\cap H_\sharp^{1+\kappa}(0, L)$ -- since both spaces are algebras -- and that $\chi \in  \BS{W}_\sharp^{1, \infty}(\Omega).$ 
 
\medskip
 
For the definition of $\vg,$ we shall transform $\wg$ into a vector field $\wg^{\chi}$ satisfying $\wg^{\chi} \cdot \BS{e}_1 = 0$ on $\mathcal S_{\hb}.$ To preserve simultaneously that $\wg^{\chi}$ is divergence free, one multiplies, in a standard way, the vector field by the cofactor of  $\nabla \chi$. So, we now analyze  the multiplier properties of $\text{Cof}(\nabla\chi)^{\top}.$ First, we have 
 \[
 \text{Cof}(\nabla\chi)^{\top} = \begin{pmatrix}
m(x)&0\\
-\partial_x m(x)(y+1)&1\\
\end{pmatrix} \in  \BS{L}^\infty_\sharp(\Omega)\cap  \BS{H}_\sharp^\kappa(\Omega).
 \]
Then, straightforward computations yield
\begin{equation} \label{prelim.h}
\|m-1\|_{W^{1,\infty}_{\sharp}(0,L)} \leq C_A \|h-\underline{h}\|_{W^{1,\infty}_{\sharp}(0,L)} ,
\end{equation}
so that, 
\begin{equation}\label{preli.C1}\norme{\text{Cof}(\nabla\chi)^{\top}-\mathbb I_2}_{\BS{L}^{\infty}_{\sharp}(\CM)}\leq C_A\norme{h-\hb}_{W^{1,\infty}_{\sharp}(0,L)},
\end{equation}
with $C_A$ a constant depending only on the upper $A$ defined by \eqref{borne:A}.
Finally we prove $H^{\sigma}$--estimates. 
To this end, we interpolate between $L^2$ and $H^\kappa$. Estimate \eqref{preli.C1} implies
\[
\norme{\text{Cof}(\nabla\chi)^{T}-\mathbb I_2}_{\BS{L}^{2}_{\sharp}(\CM)}\leq C_A\norme{h-\hb}_{W^{1,\infty}_{\sharp}(0,L)}.
\]
Then, we remark that
\[
\norme{m-1}_{H^{1+\kappa}_{\sharp}(0,L)}\leq C_{A},
\]
and that, for any given $f$, $a$ and $b$ regular functions defined on $(0,L),$ there holds
\begin{equation}\label{2D.1D.estimate}\norme{f}_{H_\sharp^{\sigma}(\mathcal{C}_{a}^{b})}\leq \norme{a-b}_{L_\sharp^{\infty}(0,L)}^{\frac{1}{2}}\norme{f}_{H_\sharp^{\sigma}(0,L)},
\end{equation}
for $0\leq \sigma\leq 1$.   Consequently, we obtain also
\[\norme{\text{Cof}(\nabla\chi)^{\top}-\mathbb I_2}_{\BS{H}^{\kappa}_{\sharp}(\CM)}\leq C_A\norme{m-1}_{H^{1+\kappa}_{\sharp}(0,L)}.
\]
Using interpolation between the $L^{2}$ and the $H^{\kappa}$ estimates finally leads to
\begin{equation}\label{preli.Hs}\norme{\text{Cof}(\nabla\chi)^{\top}-\mathbb I_2}_{\BS{H}^{\sigma}_{\sharp}(\CM)}\leq C_A \norme{h-\hb}_{W^{1,\infty}_{\sharp}(0,L)}^{\frac{\kappa-\sigma}{\kappa}},
\end{equation}
for $0\leq \sigma\leq \kappa$.

\medskip

\noindent\textbf{{Step 2. Construction of $(\vg,d)$.}}  Let consider $(\wg,\etap) \in X^s[h]$ enjoying the further property
\begin{eqnarray}
&\wg_{\vert\mathcal{F}_{h}^{-}}\in \BS{H}^{1}_{\sharp}(\mathcal{F}^{-}_{h}).\label{hyp-Ps2-proof}
\end{eqnarray}
As mentioned previously, to define $(\vg,d)$ we first construct an intermediate vector field 
$\wg^{\chi}$ obtained by the change of variables $\chi$ from $\wg.$ However we note  that 
$\chi$ does not map $\Omega$ into $\Omega$ so that we must at first extend $\wg$ for
$y\geq 2 M.$ Namely, we set
\[
\tilde{\wg}(x,y) = 
\left\{
\begin{array}{ll}
\wg(x,y), & \text{ if $(x,y) \in \Omega$}, \\
\etap(x)\BS{e}_{2}, & \text{ if $y >2M$}.
\end{array}
\right. 
\] 
This extension preserves the divergence free constraint.
We next define
\[\wg^\chi = \text{Cof}(\nabla \chi)^{\top}\tilde{\wg}\circ \chi.\] 
The $\text{Cof}(\nabla \chi)^{\top}$ factor ensures that  $\wg^\chi$ is also divergence free.
Next we define $(\vg,d)$ as
\begin{equation}\label{def.v.b}
\vg=\begin{cases}
\begin{array}{ll}
\wg^\chi -{w_{2}^\chi}_{|_{y=0}}\BS{e}_{2}, & \text{ in } \mathcal{C}^{2M}_{0},\\
0, & \text{ in }\mathcal{C}^{0}_{-1},
\end{array}
\end{cases}
\qquad
d=\etap -{w_{2}^\chi}_{|_{y=0}}.
\end{equation}
The first step is to verify that  $(\vg, d) \in X^s[\hb]$, $\forall s<\kappa/2$.  First,  by taking into account  assumption \eqref{hyp-Ps2-proof}, since $\tilde{\wg}$ coincides with $\wg$ in $\mathcal F_{h}^{-}$, we have that $\tilde{\wg}\in H_\sharp^1(\mathcal F_{h}^{-})$. Thus since the change of variables  $\chi$ maps $\mathcal{F}^-_{\hb}$ onto $\mathcal{F}^-_{h},$ the above analysis of the  regularity of $\chi$ and of $ \text{Cof}(\nabla \chi)^{\top}$ implies that $\wg^\chi\in \BS{H}^s_\sharp(\mathcal{F}^-_{\hb})$, $\forall s<\kappa$ (see \cite[Proposition B.1]{Grubsolo} ).  Moreover by the change of variables, the boundary $y=0$ is mapped to $y= m-1 $ which is lower than $h$ and strictly greater that $-1$. Hence, the trace of  $\tilde{\wg} \circ \chi$ on $y=0$ is well defined and belongs to $ H^{1/2}_{\sharp}(0,L)$. 
But by definition we have
\[
{w_{2}^\chi}_{|_{y=0}} = -\partial_x m \  \tilde{w}_{1}\circ\chi_{|_{y=0}}+\tilde{w}_{2}\circ\chi_{|_{y=0}},
\]
where $\partial_x m\in H^\kappa_\sharp(0, L).$ Classical multiplier arguments  thus  imply that ${w_{2}^\chi}_{|_{y=0}} \in H^{2s}_{\sharp}(0,L)$  for any  $s < \kappa/2$  and that 
\begin{equation} \label{eq.lapremieresurw2chi}
\|{w_{2}^\chi}_{|_{y=0}}\|_{H^{2s}_{\sharp}(0,L)}  \leq C_A(\|h-\hb\|_{W^{1,\infty}_{\sharp}(0,L)}) \|\wg\|_{H^1(\mathcal F^{-}_h)},
\end{equation}
where $C_{A}(x)\underset{x\rightarrow 0}{\longrightarrow}0$.
Furthermore we have by construction that 
\[
\text{
\qquad $\wg^\chi= \etap \BS{e}_{2}$ in  $\mathcal {S}_\hb$.
}
\]
Consequently, thanks to the regularity of $\etap$ and the one obtained on  ${{w}^\chi_{2}}_{|_{y=0}}$, we deduce that $\vg\in \BS{H}^{2s}(\mathcal {S}_\hb)\subset \BS{H}^{s}(\mathcal {S}_\hb)$.
Finally $(\vg, d)\in \BS{H}^s(\Omega)\times H^{2s}(0, L)$, for $s<\kappa/2$.
Let us now check the divergence free constraint and the fluid--structure velocity matching.
We have by construction that 
\[
\text{
$\text{div}\,\wg^\chi =0$ in $\Omega$.}
\]
Thus $\vg$ satisfies
\[
\text{
$\text{div}\,\vg =0$ in $\mathcal{C}^{2M}_{0}$  and $\text{div}\,\vg =0$ in $\mathcal{C}^{-1}_{0}$.
 }
\]
Since, by construction $\vg_{|_{y=0}}=0$, we obtain $\text{div}\,\vg =0$ in $\Omega$.
Moreover ${\rm div}\, \wg^{\chi} = 0$ on $\mathcal C_{-1}^0$ with $\wg^{\chi} = 0$ on $y=-1.$ 
By integrating this divergence constraint we obtain the condition
\[
\int_{0}^L {w_{2}^\chi}_{|_{y=0}} = 0.
\]
As a consequence, since $\dot\eta\in L^2_{\sharp,0}(0,L)$, we obtain   $d \in L^2_{\sharp,0}(0,L).$ 

We now check  the remaining compatibility conditions of $X^s[h].$ For  $y \geq \hb,$ we have $\vg = ( \dot{\eta} - {w_{2}^\chi}_{|_{y=0}}) \BS{e}_2$ so that $\vg$ satisfies
\[
\vg \cdot \BS{e}_1 = 0\,, \text{ on $\mathcal S_{\hb}$ }, \qquad  {v_2}_{|_{y=M}} = d. 
\]
This ends the proof that $(\vg,d) \in X^s[\hb].$

%
%
%
%

\medskip

\noindent{\textbf{Step 3. Splitting of  $\|(\wg,\etap) - (\vg,d)\|_{X^s}$.}} 
Let first remark that
\[\vg - \wg^\chi =\begin{cases}
\begin{array}{ll}
-{w_{2}^\chi}_{|_{y=0}}\BS{e}_{2},&\text{ in  } \mathcal{C}^{2M}_{0},\\
-\wg^\chi, & \text{ in   }\mathcal{C}^{0}_{-1},
\end{array}
\end{cases}
\qquad
\etap -d ={w_{2}^\chi}_{|_{y=0}}.
\]
Consequently we have
\[
\begin{aligned}
\norme{(\wg,\etap) - (\vg,d)}_{X^s}&\leq \norme{(\wg,\etap) - (\wg^\chi,d)}_{X^s} +\norme{\wg^\chi - \vg}_{\BS{H}^{s}_{\sharp}(\CM)}\\
&\leq \norme{\wg -\wg^\chi}_{\BS{H}^{s}_{\sharp}(\CM)} + \norme{{w_{2}^\chi}_{|_{y=0}}}_{H^{2s}_{\sharp}(0,L)}+ \norme{{w_{2}^\chi}_{|_{y=0}}\BS{e}_{2}}_{\BS{H}^{s}_{\sharp}(\mathcal{C}^{2M}_{0})}+\norme{\wg^\chi}_{\BS{H}^{s}_{\sharp}(\mathcal{C}_{-1}^{0})}.
\end{aligned}
\]
Recalling \eqref{2D.1D.estimate} we obtain the bound
\[\norme{{w_{2}^\chi}_{|_{y=0}}\BS{e}_{2}}_{\BS{H}_\sharp^{s}(\mathcal{C}_{0}^{2M})}\leq \sqrt{2M}\norme{{w_{2}^\chi}_{|_{y=0}}\BS{e}_{2}}_{\BS{H}_\sharp^{s}(0,L)}\leq \sqrt{2M}\norme{{w_{2}^\chi}_{|_{y=0}}}_{H_\sharp^{2s}(0,L)}.\]
Moreover, as $\wg=0$ in $\mathcal{C}^{0}_{-1},$ we remark that an estimate on $\norme{\wg^\chi - \wg}_{\BS{H}^{s}(\CM)}$ implies an estimate on $\norme{\wg^\chi}_{\BS{H}^{s}(\mathcal{C}_{-1}^{0})}$. Finally \eqref{est.projector.main} is implied by the following estimate
\begin{equation}\label{est.projector.final}
\norme{\wg -\wg^\chi}_{\BS{H}^{s}_{\sharp}(\CM)}+ \norme{{w_{2}^\chi}_{|_{y=0}}}_{H^{2s}_{\sharp}(0,L)}
\leq C_{A}(\norme{\hb -h }_{W^{1,\infty}_{\sharp}(0,L)})\norme{\wg}_{\BS{H}^{1}_{\sharp}(\mathcal{F}^{-}_{h})}.
\end{equation}
Thus we have to prove that $\wg -\wg^\chi=\wg-\text{Cof}(\nabla \chi)^{\top}\wg\circ \chi $ and  ${w_{2}^\chi}_{|_{y=0}}=-\partial_x m \  w_{1}\circ\chi_{|_{y=0}}+w_{2}\circ\chi_{|_{y=0}}$ can be estimated with respect to the difference $h-\hb$.  This is the  aim of the two next steps respectively.

\medskip

\textbf{Step 4. {Estimating $\wg -\wg^\chi.$}} 
We estimate the difference $\wg - \wg^{\chi}$ 
by considering successively each of the subdomains of $\Omega$: $\mathcal{S}_{h}$, $\mathcal{C}_{\hb}^{h}$,  $\mathcal{F}_{\hb}$, $\mathcal{C}_{-1}^{0}.$ 

\noindent\textit{Estimates in $\mathcal{S}_{h}$.} In $\mathcal{S}_{h}$,  $\wg=\etap \BS{e}_2.$ By replacing in the definition of $\wg^{\chi}$, we have also $\wg^\chi = \etap\BS{e}_2$ in $\mathcal{S}_{\hb}$ and since $h\geq \hb$ we infer $\wg - \wg^{\chi} = 0$ in $\mathcal{S}_{h}$. \\

\noindent{\textit{Estimates in $\mathcal{C}^{h}_{\hb}$.}} The identity $\wg^\chi=\etap\BS{e}_{2}$ still  holds in $\mathcal{C}^{h}_{\hb} \subset \mathcal S_{\hb}$ which leads to
\[\wg(x,y)- \wg^{\chi}(x,y) = \wg(x,y) - \wg(x,h(x)) =\int_{y}^{h(x)}\partial_y\wg(x,z)\dd z.\]
We obtain then
\[
\begin{aligned}
\norme{\wg^\chi - \wg}_{\BS{L}^{2}_{\sharp}(\mathcal{C}_{\hb}^{h})}^{2}&{}=\int_{0}^{L}\int_{\hb(x)}^{h(x)}\left\vert \int_{y}^{h(x)}\partial_y\wg(x,z)dz\right\vert^{2}\dd y\dd x\\
&\leq \int_{0}^{L}\int_{\hb(x)}^{h(x)}\norme{h-\hb}_{L^{\infty}_{\sharp}(0,L)}\int_{\hb(x)}^{h(x)}\vert\partial_y\wg(x,z)\vert^{2}\dd z\dd y\dd x\\
&\leq \norme{h-\hb}_{L^{\infty}_{\sharp}(0,L)}^{2}\norme{\nabla\wg}_{\BS{L}^{2}_{\sharp}(\mathcal{F}^{-}_{h})}^{2}.
\end{aligned}
\]
Moreover
\[\norme{\wg^\chi-\wg}_{\BS{H}^{1/2}_{\sharp}(\mathcal{C}_{\hb}^{h})}\leq \norme{\wg^\chi}_{\BS{H}^{1/2}_{\sharp}(\mathcal{C}_{\hb}^{h})} + \norme{\wg}_{\BS{H}^{1/2}_{\sharp}(\mathcal{C}_{\hb}^{h})}\leq \norme{\wg^\chi}_{\BS{H}^{1/2}_{\sharp}(\mathcal{C}_{\hb}^{h})} + \norme{\wg}_{\BS{H}^{1}_{\sharp}(\mathcal{F}^{-}_{h})}.
\]
We have 
$\wg^\chi=\etap\BS{e}_2=\wg_{|_{y=h}}$. Hence,
recalling the trace continuity estimate \eqref{est:trace}, we obtain
$$
\norme{\wg^\chi}_{\BS{H}^{1/2}_{\sharp}(0,L)} \leq C_A\norme{\wg}_{H^{1}_{\sharp}(\mathcal{F}^{-}_{h})}.
$$
We conclude that
\[\norme{\wg^\chi-\wg}_{\BS{H}^{1/2}_{\sharp}(\mathcal{C}_{\hb}^{h})}\leq C_A\norme{\wg}_{\BS{H}^{1}_{\sharp}(\mathcal{F}^{-}_{h})},\]
and by interpolation with the previous $L^2$-estimate, the following estimate in $H^s$ holds true
\begin{equation} \label{eq_enhaut}
\norme{\wg^\chi-\wg}_{\BS{H}^{s}_{\sharp}(\mathcal{C}^{h}_{\hb})}\leq C_A\norme{h-\hb}_{W^{1,\infty}_{\sharp}(0,L)}^{1-2s}\norme{\wg}_{\BS{H}^{1}_{\sharp}(\mathcal{F}^{-}_{h})}.
\end{equation}

\noindent{\textit{Estimates in $\mathcal{F}_{\hb}$:}} Consider the following splitting:
\begin{equation}\label{splitting.Du}
\wg^\chi - \wg = \text{Cof}(\nabla\chi)^{\top}(\wg\circ\chi - \wg) +(\text{Cof}(\nabla\chi)^{\top}-\mathbb I_2)\wg.
\end{equation}
Thanks to \cite[Proposition B.1]{Grubsolo} we obtain, for $s<s'\leq \kappa$ (see  \eqref{preli.Hs} for the estimate of the $H^{s'}$-norm of the cofactor matrix): 
\[\begin{aligned}
\norme{(\text{Cof}(\nabla\chi)^{\top}-\mathbb I_2)\wg}_{ \BS{H}^{s}_{\sharp}(\mathcal{F}_{\hb})}&\leq C_{A} \norme{(\text{Cof}(\nabla\chi)^{\top}-\mathbb I_2)}_{ \BS{H}^{s'}_{\sharp}(\CM)}\norme{\wg}_{\BS{H}^{1}_{\sharp}(\mathcal{F}^{-}_{h})}\\
&\leq  C_A \norme{h-\hb}_{W^{1,\infty}_{\sharp}(0,L)}^{\frac{\kappa-s'}{\kappa}}\norme{\wg}_{\BS{H}^{1}_{\sharp}(\mathcal{F}^{-}_{h})},\end{aligned}\]
Here we use  the continuity of the 
multiplication  
$H^{s'}(\mathcal F^{-}_{\hb}) \times H^{1}(\mathcal F^{-}_{h}) \to
H^{s}(\mathcal F^{-}_{\hb}).$ The  continuity constant of this mapping
may depend on $\hb.$ But, by  a standard change of variables argument, we see that it depends increasingly on $\|\hb\|_{W^{1,\infty}_{\sharp}(0,L)}$ only. This constant thus depends on $A$ only.
We now take care of the first term of the right-hand side of \eqref{splitting.Du}.
Let first note that we can bound the $L^2$--norm of $\wg \circ \chi - \wg$ as follows:
\[
\begin{aligned}
\int_{0}^{L}\int_{0}^{\hb(x)}\vert\wg(\chi(x,y)) - \wg(x,y)\vert^{2}\dd y\dd x&{}=\int_{0}^{L}\int_{0}^{\hb(x)}\vert \wg(x,m(x)(y+1)-1)-\wg(x,y)\vert^{2}\dd y\dd x\\
&\leq \int_{0}^{L}\int_{0}^{\hb(x)}\left\vert \int_{y}^{m(x)(y+1)-1}\partial_y\wg(x,z)dz\right\vert^{2}\dd y\dd x\\
&\leq \int_{0}^{L}\int_{0}^{\hb(x)}(m(x)-1)(y+1)\int_{y}^{m(x)(y+1)-1}\vert \partial_y\wg(x,z)\vert^{2}\dd z\dd y\dd x\\
&\leq (M+1)\norme{m-1}_{L^{\infty}_{\sharp}(0,L)}\norme{\hb}_{L^{\infty}_{\sharp}(0,L)}\norme{\nabla\wg}_{\BS{L}^{2}_{\sharp}(\mathcal{F}^{-}_{h})}^{2}.
\end{aligned}
\]
The previous estimate leads to
\[
\begin{aligned}
\norme{\wg\circ\chi-\wg}_{\BS{L}^{2}_{\sharp}(\mathcal{F}_{\hb})}&{}\leq \sqrt{(M+1)M}\norme{m-1}_{L^{\infty}_{\sharp}(0,L)}^{1/2}\norme{\wg}_{\BS{H}^{1}_{\sharp}(\mathcal{F}_{h}^{-})}\nonumber\\
&\leq C_M\norme{h-\hb}_{W^{1,\infty}_{\sharp}(0,L)}^\frac{1}{2}\norme{\wg}_{\BS{H}^{1}_{\sharp}(\mathcal{F}_{h}^{-})}.
\end{aligned}
\]
Finally, since $\norme{\text{Cof}(\nabla\chi)^{\top}}_{L^\infty_\sharp(\mathcal{F}_{\hb})}\leq C_A$, we deduce
\begin{equation}
\label{est:L2:Fh}
\norme{\text{Cof}(\nabla\chi)^{\top}(\wg\circ\chi-\wg)}_{\BS{L}^{2}_{\sharp}(\mathcal{F}_{\hb})}\leq C_A\norme{h-\hb}_{W^{1,\infty}_{\sharp}(0,L)}^\frac{1}{2}\norme{\wg}_{\BS{H}^{1}_{\sharp}(\mathcal{F}_{h}^{-})}.
\end{equation}
Next we remark that  $\wg\circ\chi - \wg$ is bounded in $\BS{H}_{\sharp}^1(\mathcal{F}^{-}_{\hb})$. Indeed $\wg\in \BS{H}_{\sharp}^1(\mathcal{F}^-_{h})$ and thus  $\wg\in \BS{H}_{\sharp}^1(\mathcal{F}^{-}_{\hb})$. It implies also that $\wg\circ\chi \in \BS{H}_{\sharp}^1(\mathcal{F}_{\hb})$ since $\chi$ belongs to  $W_{\sharp}^{1, \infty}(\Omega)$ and maps $\mathcal{F}^-_{\hb}$ in $\mathcal{F}^-_{h}$. Consequently, we have
$$
\norme{\wg\circ\chi - \wg}_{\BS{H}^{1}_{\sharp}(\mathcal{F}_{\hb})}\leq  \norme{\wg\circ\chi}_{\BS{H}^{1}_{\sharp}(\mathcal{F}_{\hb})} + \norme{\wg}_{\BS{H}^{1}_{\sharp}(\mathcal{F}_{\hb})}  \leq C_{A}\norme{\wg}_{\BS{H}^{1}_{\sharp}(\mathcal{F}^{-}_{h})}.
$$
Next, thanks to the fact that $\norme{\text{Cof}(\nabla\chi)^{\top}}_{\BS{H}^\kappa_\sharp(\mathcal{F}_{\hb})}\leq C_A$ (see \eqref{preli.Hs}), we have, since  $0\leq s<\kappa$ 
\begin{equation}
\norme{\text{Cof}(\nabla\chi)^{\top}(\wg\circ\chi - \wg)}_{\BS{H}^{s}_{\sharp}(\mathcal{F}_{\hb})}\leq C_{A}{\norme{\text{Cof}(\nabla\chi)^{\top}}_{\BS{H}^{\kappa}_{\sharp}(\Omega)}}\norme{\wg}_{\BS{H}^{1}_{\sharp}(\mathcal{F}^{-}_{h})}\leq C_{A}\norme{\wg}_{\BS{H}^{1}_{\sharp}(\mathcal{F}^{-}_{h})}.\label{est:H1:Fh}
\end{equation}
By interpolating \eqref{est:L2:Fh} and \eqref{est:H1:Fh}, we obtain
\[
\norme{\text{Cof}(\nabla\chi)^{\top}(\wg\circ\chi - \wg)}_{\BS{H}^{s}_{\sharp}(\mathcal{F}_{\hb})}\leq C_A(\norme{h-\hb}_{W^{1,\infty}_{\sharp}(0,L)})\norme{\wg}_{\BS{H}^{1}_{\sharp}(\mathcal{F}^{-}_{h})}.
\]
To summarize the estimates in $\mathcal{F}_{\hb}$ we have proved that
\begin{equation} \label{eq.milieu}
\norme{\wg^\chi-\wg}_{\BS{H}^{s}_{\sharp}(\mathcal{F}_{\hb})}\leq C_A(\norme{h-\hb}_{W^{1,\infty}_{\sharp}(0,L)})\norme{\wg}_{\BS{H}^{1}_{\sharp}(\mathcal{F}^{-}_{h})}.
\end{equation}

\noindent{\textit{Estimates in $\mathcal{C}_{-1}^{0}$.}} The function $\wg$ is equal to zero and we have to estimate only $\wg^\chi$. As previously we obtain a first bound in $L^{2}$ involving $C_A(\norme{h-\hb}_{W^{1,\infty}_{\sharp}(0,L))})$ and then we prove that $\wg^\chi$ is bounded in some $H^{s}$ and we conclude using interpolation. For the $L^{2}$--norm we have
\[
\norme{\wg^\chi}_{\BS{L}^{2}_{\sharp}(\mathcal{C}_{-1}^{0})}\leq \norme{(\text{Cof}\nabla\chi)^{\top}}_{\BS{L}^{\infty}_{\sharp}(\mathcal{C}_{-1}^{0})}\norme{\wg\circ\chi}_{\BS{L}^{2}_{\sharp}(\mathcal{C}_{-1}^{0})},
\]
and
\[
\begin{aligned}
\norme{\wg\circ\chi}^2_{\BS{L}^{2}_{\sharp}(\mathcal{C}_{-1}^{0})}=&\int_{0}^{L}\int_{-1}^{0}\vert \wg(\chi(x,y))\vert^{2}\dd y\dd x&{}\\
&=\int_{0}^{L}\int_{0}^{m(x)-1}\vert \wg(x,y)\vert^{2}\frac{dy}{m(x)}\dd x\\
&\leq (1+M)\int_{0}^{L}\int_{0}^{m(x)-1}\left\vert\int_{0}^{y}\partial_y\wg(x,z)dz\right\vert^{2}\dd y\dd x\\
&\leq (1+M)\int_{0}^{L}\int_{0}^{m(x)-1}y\int_{0}^{y}\vert\partial_z\wg(x,z)\vert^{2}\dd z\dd y\dd x\\
&\leq \frac{(1+M)}{2}\norme{m-1}_{L^{\infty}_{\sharp}(0,L)}^{2}\norme{\wg}_{\BS{H}^{1}_{\sharp}(\mathcal{F}^{-}_{h})}^{2}.
\end{aligned}
\]
Hence, using the estimates above with \eqref{prelim.h}-\eqref{preli.C1}, we conclude
\[\norme{\wg^\chi}_{\BS{L}^{2}_{\sharp}(\mathcal{C}_{-1}^{0})}\leq C_A\norme{h-\hb}_{W^{1,\infty}_{\sharp}(0,L)}\norme{\wg}_{\BS{H}^{1}_{\sharp}(\mathcal{F}^{-}_{h})}.
\]
Moreover, for any $0\leq \sigma<\kappa$, we have
\[
\begin{aligned}
\norme{\wg^\chi}_{\BS{H}^{\sigma}_{\sharp}(\mathcal{C}_{-1}^{0})}&{}\leq C_{1}\norme{(\text{Cof}\nabla\chi)^{\top}}_{\BS{H}^{\kappa}_{\sharp}(\mathcal{C}_{-1}^{0})}\norme{\wg\circ\chi}_{\BS{H}^{1}_{\sharp}(\mathcal{C}_{-1}^{0})}\\
&\leq C_A\norme{\wg}_{\BS{H}^{1}_{\sharp}(\mathcal{F}^{-}_{h})},
\end{aligned}
\]
and using interpolation up to choose $\sigma \in (s,\kappa)$
\[\norme{\wg^\chi}_{\BS{H}^{s}_{\sharp}(\mathcal{C}_{-1}^{0})}\leq C_A(\norme{h-\hb}_{W^{1,\infty}_{\sharp}(0,L)})\norme{\wg}_{\BS{H}^{1}_{\sharp}(\mathcal{F}^{-}_{h})}.\]
To summarize we have proved that
\begin{equation} \label{eq.dessous}
\norme{\wg-\wg^\chi}_{\BS{H}^{s}_{\sharp}(\mathcal C_{-1}^0)}\leq C_A(\norme{h-\hb}_{W^{1,\infty}_{\sharp}(0,L)})\norme{\wg}_{\BS{H}^{1}_{\sharp}(\mathcal{F}^{-}_{h})}.
\end{equation}
Finally, combining \eqref{eq_enhaut}-\eqref{eq.milieu}-\eqref{eq.dessous}, we obtain  the expected estimate
\begin{equation} \label{eq.total}
\norme{\wg-\wg^\chi}_{\BS{H}^{s}_{\sharp}(\CM)}\leq C_A(\norme{h-\hb}_{W^{1,\infty}_{\sharp}(0,L)})\norme{\wg}_{\BS{H}^{1}_{\sharp}(\mathcal{F}^{-}_{h})}.
\end{equation}

\noindent\textbf{{Step 5. Estimating ${{w_{2}^\chi}_{|_{y=0}}}.$}} 
First, we recall that
\[
{w_{2}^\chi}_{|_{y=0}}=-\partial_x m(\cdot)w_{1}(\cdot,m(\cdot)-1)+w_{2}(\cdot,m(\cdot)-1).
\]
This term is first estimated in $L^2$, then in $H^{2\sigma}$ for $0<\sigma< \kappa/2$, and the final estimate is obtained   by interpolation. First let estimate the $L^{2}$--norm
\[
\begin{aligned}
\norme{-\partial_x m(\cdot)w_{1}(\cdot,m(\cdot)-1)}_{L^{2}(0,L)}^{2}&{}\leq \norme{\partial_x m}_{L^{\infty}_{\sharp}(0,L)}^{2}\int_{0}^{L}\vert w_{1}(x,m(x)-1)\vert^{2}\dd x\\
&\leq C_{A}\int_{0}^{L}(m(x)-1)\int_{0}^{m(x)-1}\vert\partial_yw_{1}(x,y)\vert^{2}\dd y\dd x\\
&\leq C_{A}\norme{m-1}_{L^{\infty}_{\sharp}(0,L)}\norme{\wg}_{\BS{H}^{1}_{\sharp}(\mathcal{F}^{-}_{h})}\\
&\leq C_{A}\norme{h-\hb}_{L^{\infty}_{\sharp}(0,L)}\norme{\wg}_{\BS{H}^{1}_{\sharp}(\mathcal{F}^{-}_{h})}.
\end{aligned}\]
A similar estimate can be computed for $\norme{w_{2}(\cdot,m(\cdot)-1)}_{L^{2}_{\sharp}(0,L)}$ so that, we obtain 
\[\norme{{w_{2}^\chi}_{|_{y=0}}}_{{L}^{2}_{\sharp}(0,L)}\leq C_A\norme{h-\hb}_{W^{1,\infty}_{\sharp}(0,L)}\norme{\wg}_{\BS{H}^{1}_{\sharp}(\mathcal{F}^{-}_{h})}.\]
For $0<\sigma< \kappa/2$ we obtain, similarly to \eqref{eq.lapremieresurw2chi}
\[
\norme{{w_{2}^\chi}_{|_{y=0}}}_{{H}^{2\sigma}_{\sharp}(0,L)} \leq C_A(\norme{h-\hb}_{W^{1,\infty}_{\sharp}(0,L)})\norme{\wg}_{\BS{H}^{1}_{\sharp}(\mathcal{F}^{-}_{h})}.\]
Using interpolation we finally obtain (up to choose $s \leq \sigma < \kappa/2$) that
\[\norme{{w_{2}^\chi}_{|_{y=0}}}_{H^{2s}_{\sharp}(0,L)}\leq C_A(\norme{h-\hb}_{W^{1,\infty}_{\sharp}(0,L)})\norme{\wg}_{\BS{H}^{1}_{\sharp}(\mathcal{F}^{-}_{h})},\]
which concludes the proof of \eqref{est.projector.final} and the proof of the lemma is completed.
\end{proof}

\section{Proof of Theorem \ref{main.theorem}}

This section is devoted to the proof of existence of weak solutions of $(FS)_0$. 
So, we fix $T>0$ and initial data $(\ug^0,\eta^0,\eta^1)$ satisfying \eqref{CI1}--\eqref{CI4}. 
We recall that the strategy is to approximate this problem by a sequence of viscous problems $(FS)_\gamma$, $\gamma>0$, for which existence results are available. The proof is divided into three steps.  First, we analyze the Cauchy theory of $(FS)_{\gamma}$ when $\gamma >0$ and prove that the sequence of solutions converges, up to a subsequence,  when $\gamma \to 0.$   We show in particular that  possible weak limits  are  candidates to be weak solutions up to obtaining 
strong compactness of approximate velocities in $L^2.$ As explained in the introduction, this strong compactness property is the cornerstone of the analysis.  Our proof builds on the projection/approximation argument  provided by \cite{SanMartin-Starovoitov-Tucsnak} 
in the fluid--solid case. In our fluid--elastic setting, it requires to build a uniform bound by below $\hb$ of the sequence
of approximate structure deformations (in order to construct a fluid domain independent of $\gamma$ on which the Navier--Stokes equations are satisfied by the sequence of approximate solutions to be able to apply Aubin--Lions Lemma for projections of the velocities). The second step of the proof is devoted to the construction of $\hb$ and
the analysis of its properties. We  then complete the proof of the $L^2$--strong compactness. This last step relies in particular on the continuity result obtained in subsection \ref{continuite_h}.

\subsection{Step 1. Construction of a candidate weak-solution.}
\label{sec:proof1}
Let us recall the strong existence result on $(0,T)$ stated in \cite[Theorem 1]{Grandmont-Hillairet}. 
Given $\gamma>0$ and initial data $(\ug_\gamma^{0},\eta_\gamma^{0},\eta_\gamma^{1})$ satisfying
\begin{align}
&(\eta_\gamma^{0},\eta_\gamma^{1})\in H^{3}_{\sharp}(0,L)\times H^{1}_{\sharp}(0,L),\label{CI-gamma1}\\
&\ug_\gamma^{0}\in \BS{H}_{\sharp}^{1}(\mathcal{F}_{h_{\gamma}^{0}}), \DIV{\ug_\gamma^{0}}=0\text{ in }\mathcal{F}_{h^{0}},\label{CI-gamma2}\\
&\ug_\gamma^{0}(x,0)=0, \text{ and }\ug_\gamma^{0}(x,h_\gamma^{0}(x))=\eta_\gamma^{1}(x)\BS{e}_{2},  \forall x\in [0,L],\label{CI-gamma3}\\
&\min_{x\in[0,L]}h_\gamma^{0}(x)>0\text{ and }\int_{0}^{L}\eta_\gamma^{1}(x)\dd x=0,\label{CI-gamma4}
\end{align}
the system $(FS)_{\gamma}$ admits a unique strong solution defined on $(0, T)$.  This solution satisfies moreover $\min_{x\in [0, L]} h_\gamma(x, t) >0$ for all $t\in [0, T]$.

\medskip

In order to apply this result we  now explain  the construction of a sequence of regular initial data $(\ug_\gamma^0,\eta_\gamma^0,  \eta_\gamma^1)_{\gamma >0}$ approximating $(\ug^0, \eta^0, \eta^1)$. First we construct $\eta^0_\gamma\in H^{3}_{\sharp}(0,L)$ by a standard convolution of $\eta^0$ with a regularizing kernel. Since $\eta^0$ satisfies \eqref{CI1}, this sequence is  uniformly bounded in $H^2_\sharp(0,L)$ and satisfies
\begin{align*}
& \eta^0_\gamma\rightarrow \eta^0, \text{ in }H^2_\sharp(0,L),\\
&  \| \eta^0_\gamma- \eta^0\|_{\mathcal{C}_\sharp^0([0, L])}\leq \gamma \|\eta^0\|_{H^2_\sharp(0, L)}\leq C\gamma.
\end{align*}
Since $\min_{x\in[0, L]}h^0(x)>0$, there exists $\lambda>0$ verifying $\min_{x\in[0, L]}h_\gamma^0(x)>\lambda >0$ for $\gamma$ small enough.
We next construct $\eta^1_\gamma \in H^{1}_{\sharp}(0,L)\cap L^2_{\sharp, 0}(0, L)$.
Since $\eta^1$ satisfies \eqref{CI2}, this second sequence enjoys the following properties:
\begin{align*}
& \eta^1_\gamma\rightarrow \eta^1, \text{ in }L^2_\sharp(0,L),\\
&  \| \eta^1_\gamma \|_{{L}^2_{\sharp,0}(0, L)}\leq \|\eta^1\|_{L^2_\sharp(0, L)}\leq C.
\end{align*}
We now build the approximate initial velocity fields $\ug^0_{\gamma}.$
A key difficulty here is to match the continuity of velocity field at the structure interface together with preserving the divergence free condition, taking into account
that the approximation is defined on an approximate domain depending on $\gamma.$
To handle this difficulty, we   first define the extension of $\eta^1$ to the whole domain using the operator $\mathcal R_{\lambda}$
as defined in \eqref{def:R}.  Next we consider $\overline{\ug^0}-\mathcal{R}_\lambda(\eta^1)$ which is in $K[h^0]$ and satisfies moreover $\overline{\ug^{0}}-\mathcal{R}_\lambda(\eta^1)= \BS{0}$ in $\mathcal{S}_{h^0}\cup \mathcal{C}_{-1}^0$.  Then we introduce the vertical contraction operator denoted by
\[
\vg \mapsto \vg_\sigma(x, y)=(\sigma v_1(x, \sigma y), v_2(x, \sigma y)) \quad \forall \, \sigma >0.
\]
We emphasize that this contraction operator preserves the divergence free constraint. By choosing 
$\sigma_{\gamma} = 1+ 2C\gamma/\lambda$ (with the constant $C$ above), we have $(\overline{\ug^{0}}-\mathcal{R}_\lambda(\eta^1))_{\sigma_\gamma}= \BS{0}$ in $\mathcal{S}_{h_\gamma^0}\cup \mathcal{C}_{-1}^0$, and that $(\overline{\ug^{0}}-\mathcal{R}_\lambda(\eta^1))_{\sigma_\gamma}$ converges  to $\overline{\ug^{0}}-\mathcal{R}_\lambda(\eta^1)$ in $\BS{L}_\sharp^2(\CM)$ when $\gamma \to 0.$ Moreover $(\overline{\ug^{0}}-\mathcal{R}_\lambda(\eta^1))_{\sigma_\gamma}$ belongs to $\BS{L}^2_{\sharp}(\mathcal{F}_{h_\gamma^0})$, is divergence free and satisfies $(\overline{\ug^{0}}-\mathcal{R}_\lambda(\eta^1))_{\sigma_\gamma}\cdot\BS{n} =0$ on $\Gamma_{h^0_\gamma}$ and $(0, L)\times\{0\}$. Thus we approximate thanks to standard arguments (by truncation and regularization of the stream function for instance)  this function by a divergence free function $(\overline{\ug^{0}}-\mathcal{R}_\lambda(\eta^1))_{\gamma}$ in $\BS{H}^1_{\sharp}(\mathcal{F}_{h_\gamma^0})$  vanishing in a neighbourhood of  $\Gamma_{h^0_\gamma}$ and $(0, L)\times\{0\}$. We may then set 
\[
\ug_{\gamma}^0 = \left( \overline{(\overline{\ug^{0}}-\mathcal{R}_\lambda(\eta^1))_{\gamma} }+\mathcal{R}_\lambda(\eta_\gamma^1)\right)_{|_{\mathcal{F}_{h_\gamma^0}}}.
\]
Straightforward computations show that $\ug_{\gamma}^0$ satisfies \eqref{CI-gamma2}-\eqref{CI-gamma3}. Moreover, 
remarking that the operator $\mathcal R_{\lambda}$ is continuous from $L^2_{\sharp,0}(0,L)$ into $L_\sharp^2(\Omega)$ we have as $\gamma$ goes to zero

\begin{align*}
& \overline{\ug^0_\gamma}\rightarrow \overline{\ug^0}, \text{ in }L^2_\sharp(\Omega),\\
&  \| \ug^0_\gamma \|_{{L}^2(\mathcal F^0_{h})}\leq C\left(  \|\eta^1\|_{L^2_\sharp(0, L)} + \|\ug^0\|_{L^2(\mathcal F_h^0)} \right)\leq C,
\end{align*}
where $C$ does not depend on $\gamma$.

We now  apply the result on existence of a strong solution for the viscous problem $(FS)_{\gamma>0}$.  For fixed $\gamma >0$ the unique solution $(\ug_\gamma,\eta_\gamma)$ is global in time so that it exists on any time interval $(0,T)$.
 The first step is to verify that  $\overline{\ug}_{\gamma}$ as defined by
\[
\overline{\ug}_{\gamma} = 
\left\{
\begin{array}{ll}
\partial_t \eta_{\gamma} \BS{e}_2, & \text{ in $\mathcal S_{h_{\gamma}},$} \\
\ug_{\gamma} , & \text{ in $\mathcal F_{h_{\gamma}},$} \\
\BS{0}, & \text{ in $\mathcal C_{-1}^0,$}
\end{array}\right.
\]
together with $\eta_{\gamma}$ is  a pair of weak solution 
to $(FS)_{\gamma}$ in the sense of Definition \ref{def.ws}. First, we note that $(\ug_\gamma,\eta_\gamma)$
satisfies estimate \eqref{energy.estimates} so that we can define a constant $M$ involved in the definition
of our weak solution framework. By construction we have that
\[
\eta_{\gamma} \in L^{\infty}(0,T;H^{2}_{\sharp}(0,L)) \cap W^{1,\infty}(0,T;L^2_{\sharp, 0}(0,L)).
\]
Moreover, \[
\overline{\ug}_{\gamma} \in  L^{\infty}(0,T ; L^2_{\sharp}(\Omega)), \quad 
\nabla \overline{\ug_{\gamma}}_{|_{\widehat{\mathcal F}_{h_{\gamma}}}} = \mathbf{1}_{\widehat{\mathcal F}_{h_{\gamma}}}\nabla \ug_{\gamma}  \in L^2(\widehat{\mathcal F}^{-}_{h_{\gamma}}),
\]
and 
Lemma \ref{lem:bar}, implies that 
\[
(\overline{\ug}_{\gamma}(t),\eta_{\gamma})(t) \in X[h_{\gamma}(t)] \text{ for a.e. $t \in (0,T)$}
 .
\] 
Thus the regularity statement $i)$  of Definition \ref{def.ws} is satisfied. Moreover the solution satisfies the kinematic condition $\ug_\gamma(t, x, 1+\eta_\gamma(t, x))=\partial_t\eta_\gamma(t,x)\BS{e}_2$ is satisfied on $(0, L)$ which implies that $\partial_t\eta_\gamma\in L^2(0, T ; H^{1/2}_\sharp(0, L))$ as the trace of  ${\overline{\ug}_{\gamma}}_{|_{\widehat{\mathcal F}_{h_\gamma}}}$. Remember here that $\min_{x\in [0, L]} h_\gamma(x, t) >0$ for all $t\in [0, T]$ so that $\mathcal{F}_{h_\gamma}$ is a Lipschitz domain and consequently   ${\overline{\ug}_{\gamma}}_{\vert_{y=h_\gamma}}$ is well defined. Thus the second item  $ii)$ of Definition  \ref{def.ws}  holds true. Then, we note that, thanks to the regularity of solutions constructed
in \cite[Theorem 1]{Grandmont-Hillairet}, the system \eqref{Navier-Stokes}--\eqref{Beam.equation} is satisfied pointwise so that
we can multiply the system with test functions $(\wg_{\gamma},d_{\gamma})$  for which  the requirements in item
$iii)$ of Definition \ref{def.ws} are  satisfied and obtain \eqref{weak.formulation.FS} after integration by parts.

Moreover, we note that the solution $(\ug_{\gamma},\eta_{\gamma})$ satisfies the energy estimate \eqref{energy.estimates} with a right hand side that converges to 
\[
\frac{1}{2}\left(\rho_{f}\int_{\mathcal{F}_{h^0}}\vert\ug^0\vert^{2} \dd \xg+ \rho_{s}\int_{0}^{L}\vert\eta^1\vert^{2}\dd x+\beta\int_{0}^{L}\vert\partial_{x}\eta^0\vert^{2}\dd x+\alpha\int_{0}^{L}\vert\partial_{xx}\eta^0\vert^{2}\dd x\right).
\]
when $\gamma \to 0.$
 Consequently,  the sequence $(\overline{\ug}_{\gamma},\eta_{\gamma})_{\gamma >0}$ satisfies the following bounds:
\begin{align}
&\overline{\ug}_{\gamma}\text{ is uniformly bounded in }\gamma\text{ in }L^{\infty}(0,T;\BS{L}_\sharp^{2}(\mathcal{F}_{h_{\gamma}(t)})),\label{bound-u-1}\\
&\|\nabla\overline{\ug}_{\gamma}\|_{L^2(\widehat{\mathcal F}^{-}_{h_{\gamma}})} \text{ is uniformly bounded in }\gamma,
\label{bound-u-2}\\
&\eta_{\gamma}\text{ is uniformly bounded in }\gamma\text{ in }L^{\infty}(0,T;H^{2}_{\sharp}(0,L))\cap W^{1,\infty}(0,T;L_\sharp^{2}(0,L)).\label{est-eta}
\end{align}

Furthemore the structure velocity $\partial_t\eta_\gamma$ is bounded uniformly  with respect to $\gamma$ is $L^2(0, T ; H^{1/2}_\sharp(0, L))$ as the trace  ${\overline{\ug}_{\gamma}}_{\vert_{y=h_{\gamma}}}$.

Finally  the sequence  $(\overline{\ug}_{\gamma})_{\gamma>0}$ satisfies additional uniform estimates  that are summarized in the following lemma:

\begin{lemma}
\label{lem:extension-u}
The sequence $(\overline{\ug}_{\gamma})_{\gamma >0}$ is uniformly bounded in $L^{4}(\CMT)$ and in 
$L^{2}(0,T;\BS{H}^{s}_{\sharp}(\CM))$ for arbitrary $s<1/2.$
\end{lemma}
\begin{proof}The bound in  $L^{2}(0,T;\BS{H}^{s}_{\sharp}(\CM))$ comes  again from Lemma \ref{lem:bar}. We next take care of the $L^{4}(\CMT)$ uniform bound.
To prove it, it is sufficient  to obtain independent uniform bounds for the restrictions
of $\overline{\ug}_{\gamma}$ to $\widehat{\mathcal S}_{h_{\gamma}}$ and to $\widehat{\mathcal F}^{-}_{h_{\gamma}}$ for the $L^4$--norm.  

In  $\widehat{\mathcal{S}}_{h_{\gamma}}$ we already know that $\overline{\ug}_{\gamma}=\partial_{t}\eta_{\gamma}\BS{e}_{2}$ where $\partial_t{\eta}_{\gamma}$ is bounded in $L^{2}(0,T;H_\sharp^{1/2}(0,L))$. Moreover $\partial_{t}\eta_{\gamma}$ is also bounded in $L^{\infty}(0,T;L_\sharp^{2}(0,L)).$ By interpolation, we obtain  that  $\partial_{t}\eta_{\gamma}$ --~and consequently resp. $\overline{\ug}_{\gamma}$~-- is uniformly bounded in $L^{4}(0,T;L_\sharp^{4}(0,L))$, resp. $L^4(\widehat{\mathcal S}_{h_{\gamma}})$. 

We would like to apply a similar interpolation argument on the domain $\mathcal F_{h_{\gamma}}^{-}$ using the uniform bounds on the restrictions of $\overline{\ug}_{\gamma}$ and its gradient in $\mathcal F_{h_{\gamma}}^{-}$. To track the dependencies with respect to $h_{\gamma}$ to be able to ensure that the interpolation argument leads to uniform bounds, we use a change of variables. We denote by
\[
\tilde{\ug}_{\gamma}(t,x,z) = \overline{\ug}_{\gamma}(t,x,(h_{\gamma}(x) + 1)z - 1 ) \quad \forall \, (t,x,z) \in (0,T) \times \mathcal{C}_0^1.
\]
Since $0 \leq h_{\gamma} \leq M,$ straightforward computations show that, for a.e. $t \in (0,T)$ we have
\[
\|\overline{\ug}_{\gamma}(t)\|_{L_\sharp^4(\mathcal F_{h_{\gamma}(t)})}  \leq (\|h_{\gamma}(t)\|_{L^{\infty}_{\sharp}(0,L)}+1)^{1/4} \|\tilde{\ug}_{\gamma}(t)\|_{L_\sharp^4( \mathcal{C}_0^1)},
\]
and 
\begin{align*}
\|\tilde{\ug}_{\gamma}(t)\|_{L_\sharp^2( \mathcal{C}_0^1)} & \leq \|\overline{\ug}_{\gamma}(t)\|_{L_\sharp^2(\mathcal F_{h_{\gamma}(t)})},  \\
\|\nabla \tilde{\ug}_{\gamma}(t)\|_{L_\sharp^2( \mathcal{C}_0^1)}  & \leq (1+ \|h_{\gamma}(t)\|_{W^{1,\infty}_{\sharp}(0,L)}) \|\nabla \overline{\ug}_{\gamma}\|_{L_\sharp^2(\mathcal F^-_{h_{\gamma}(t)})}.
\end{align*}
However, the following interpolation inequality holds true
\[
\|\tilde{\ug}_{\gamma}(t)\|_{L_\sharp^4( \mathcal{C}_0^1)} \leq C \|\tilde{\ug}_{\gamma}(t)\|^{1/2}_{L_\sharp^2( \mathcal{C}_0^1)}\|\nabla \tilde{\ug}_{\gamma}(t)\|^{1/2}_{L_\sharp^2( \mathcal{C}_0^1)},
\]
and thus
\[
\|\overline{\ug}_{\gamma}(t)\|_{L_\sharp^4(\mathcal F^{-}_{h_{\gamma}(t)})} \leq C(1+ \|h_{\gamma}(t)\|_{W^{1,\infty}_{\sharp}(0,L)})  \|{\ug}_{\gamma}(t)\|^{1/2}_{L_\sharp^2(\mathcal F^{-}_{h_{\gamma}(t)})}\|\nabla \tilde{\ug}_{\gamma}(t)\|^{1/2}_{L_\sharp^2(\mathcal{C}_0^1)}.
\]
Applying \eqref{bound-u-1} and \eqref{bound-u-2} together with the uniform bound for 
$h_{\gamma}$ in $L^{\infty}(0,T;W^{1,\infty}_{\sharp}(0,L))$ coming from \eqref{est-eta}, we obtain the desired bound on $\overline{\ug}_{\gamma}$ which is uniformly bounded in  $L^4(\mathcal F^{-}_{h_{\gamma}}).$ \end{proof}

We now prove  the existence of cluster points of the sequence $(\overline{\ug}_{\gamma},\eta_{\gamma})_{\gamma >0}.$ First, thanks to \eqref{est-eta},  and to the compact embedding \eqref{inclusion-eta}
$$
\begin{array}{ll}
\eta_{\gamma}\rightarrow\eta& \text{ uniformly in }\mathcal{C}^{0}([0,T];\mathcal{C}^{1}_{\sharp}(0,L)) ,\\
\eta_{\gamma}\rightharpoonup\eta& \text{ weakly} -\star \text{ in }W^{1,\infty}(0,T;L^2_{\sharp}(0,L)).
\end{array}
$$

 Next, using the energy estimates \eqref{energy.estimates} and Lemma \ref{lem:extension-u}, we may construct a divergence free function  $\overline{\ug}\in L^{\infty}(0,T;\BS{L}^{2}_{\sharp}(\CM))\cap L^{4}(\CMT)$, such that, up to a subsequence,  the following convergences hold:
$$
\begin{array}{ll}
 \overline{\ug}_{\gamma}\rightharpoonup\overline{\ug},& \text{ weakly}-\star \text{ in }L^{\infty}(0,T;\BS{L}^{2}_{\sharp}(\CM)),\\
\overline{\ug}_{\gamma}\rightharpoonup\overline{\ug},&\text{ weakly in } L^{4}(\CMT).
\end{array}
$$

We now verify  that any cluster point $(\overline{\ug}, \eta)$ of the sequence $(\ug_{\gamma},\eta_{\gamma})_{\gamma >0}$ enjoys the properties of  Definition \ref{def.ws}, which defines the weak solutions of the limit coupled system $(FS)_0.$ For simplicity, we do not relabel the sequence converging to $(\overline{\ug}, \eta).$ We first note that, for fixed $\gamma >0,$ we have
\[
\int_0^L \partial_t\eta_{\gamma}=0.
\]
This property is conserved in the weak limit so that $\partial_t\eta$ is mean free globally in time. Next, we verify that $(\overline{\ug}(t),\partial_t \eta) \in X[h(t)]$ for a.e. $t \in (0,T).$ The divergence free condition is verified at the limit. Moreover, we note that,
the property $\overline{\ug}_{\gamma} = 0$ on $\widehat{\mathcal C}_{-1}^0$ is preserved in the weak limit. Furthermore, since $\eta_\gamma$ uniformly converges towards $\eta$, we easily obtain  $\overline{\ug} =\partial_{t}\eta\BS{e}_2,\text{ in }\widehat{\mathcal{S}}_{h}$, by testing the weak convergence of $ \overline{\ug}_{\gamma}$  and $ \partial_{t}\eta_\gamma\BS{e}_2$, which are equal in  $\widehat{\mathcal S}_{h_\gamma}$, against
functions $\BS{\varphi} \in \mathcal C^{\infty}_c(\widehat{\mathcal S}_h)$.


Consequently, $(\overline{\ug}(t),\partial_t \eta) \in X[h(t)].$   We now prove that $\overline{\ug}$ has better regularity in  $\widehat{\mathcal F}^-_{h}$ as stated in the following lemma:

\begin{lemma}\label{lemma.gradient}
We have $\nabla  \overline{\ug} \in L^2(\widehat{\mathcal F}_h)$ and the sequence $\rho_\gamma^-\nabla \overline{\ug}_{\gamma}$ converges to $\rho^-\nabla  \overline{\ug}$ weakly in $L^{2}(\CMT).$
\end{lemma}

\begin{remark}
In the previous statement, we use the convention that if $\mathcal O \subset \widehat{\Omega}$ and $f \in L^2(\mathcal O)$ then $\mathbf{1}_{\mathcal O} f \in L^2(\widehat{\Omega})$ is the extension by $0$ of this $L^2(\mathcal O)$--function. 
\end{remark}
\begin{proof}
We remind that $\nabla \overline{\ug}_{\gamma} \in \BS{L}^2(\widehat{\mathcal F}^{-}_{h_{\gamma}})$ so that $\rho_\gamma^-\nabla \overline{\ug}_{\gamma}$ corresponds to the extension by $\BS{0}$ of this vector field. Because of \eqref{bound-u-2}, $\rho_\gamma^-\nabla \overline{\ug}_{\gamma}$  is uniformly bounded in $\BS{L}^{2}(\widehat{\Omega})$.  Thus $\rho^-_\gamma \nabla \overline{\ug}_{\gamma}$ converges weakly to some $\bf{z}$ in $L^{2}(0,T;\BS{L}_\sharp^{2}(\CM))$. Thanks to the uniform convergence of $h_{\gamma}$ to $h$, we may then compute $\bf{z}$ by testing
the weak convergence of $\rho^-_\gamma \nabla \overline{\ug}_{\gamma}$ against
functions $\BS{\varphi} \in \mathcal C^{\infty}_c(\widehat{\mathcal S}_h)$ and 
$\BS{\varphi} \in \mathcal C^{\infty}_c(\widehat{\mathcal F}_h^{-})$ respectively. This implies that ${\bf z}_{\vert\widehat{\mathcal{S}}_h}=\BS{0}$  and 
$\BS{z}_{\vert\widehat{\mathcal{F}}^-_{h}}=(\nabla \overline{\ug})_{\vert\widehat{\mathcal{F}}^-_{h}}$, which ends the proof.
\end{proof}
\begin{remark}
The previous lemma gives the $H^1$ space regularity of $\overline{\ug}_{|_{\mathcal{F}^-_{h(t)}}}$  (for a.e. $t \in (0,T)$). Since $\mathcal{F}^-_{h(t)}$ a Lipschitz domain, it enables us to define an $H^{1/2}$ trace of $\overline{\ug}$ on ${\partial \mathcal F^-_{h(t)}}$. 
\end{remark}

This concludes the proof that $(\overline{\ug},\eta)$ satisfies  item $i)$ of  Definition
\ref{def.ws}.
{We  also prove at first that the weak cluster point
satisfies the expected energy estimate. Indeed, for  any arbitrary small $\varepsilon >0$, thanks to the strong convergence of $\eta_{\gamma}$ to $\eta,$ we have that $h_{\gamma} > h- \varepsilon$ for $\gamma$ sufficiently small. We may apply then classical weak limit arguments to pass to the limit in the energy estimate satisfied by the $({\ug}_{\gamma},\eta_{\gamma}).$ Consequently for almost every $t$ we have
\[
 \begin{aligned}
&\frac{1}{2}\left(\rho_{f}\int_{\mathcal{F}^{-}_{h(t)-\varepsilon}}\vert\ug(t, \xg) \vert^{2} \dd\xg + \rho_{s}\int_{0}^{L}\vert\partial_{t}\eta(t, x) \vert^{2} \dd x+\beta\int_{0}^{L}\vert\partial_{x}\eta(t, x) \vert^{2} \dd x+\alpha\int_{0}^{L}\vert\partial_{xx}\eta(t, x) \vert^{2} \dd x\right)\\
& + \mu\int_0^t\int_{\mathcal{F}^{-}_{h(s)-\varepsilon}}\vert \nabla\ug(s, \xg) \vert^{2} \dd\xg\dd s  \\
& \leq 
\liminf_{\gamma \to 0} \left[ \frac{1}{2}\left(\rho_{f}\int_{\mathcal{F}^{-}_{h(t)-\varepsilon}}\vert\ug_{\gamma}(t, \xg) \vert^{2} \dd\xg + \rho_{s}\int_{0}^{L}\vert\partial_{t}\eta_{\gamma}(t, x) \vert^{2} \dd x+\beta\int_{0}^{L}\vert\partial_{x}\eta_{\gamma}(t, x) \vert^{2} \dd x \right.\right. \\
&\left. \left. +\alpha\int_{0}^{L}\vert\partial_{xx}\eta_{\gamma}(t, x) \vert^{2} \dd x\right) +  \mu\int_0^t\int_{\mathcal{F}^{-}_{h(s)-\varepsilon}}\vert \nabla\ug_{\gamma}(s, \xg) \vert^{2} \dd\xg\dd s\right] \\
&  
\leq \lim_{\gamma \to 0} 
\frac{1}{2}\left(\rho_{f}\int_{\mathcal{F}_{h^0_{\gamma}}}\vert\ug^0_{\gamma}\vert^{2} \dd \xg+ \rho_{s}\int_{0}^{L}\vert\eta^1_{\gamma}\vert^{2}\dd x+\beta\int_{0}^{L}\vert\partial_{x}\eta^0_{\gamma}\vert^{2}\dd x+\alpha\int_{0}^{L}\vert\partial_{xx}\eta^0_{\gamma}\vert^{2}\dd x\right)
\\
& 
=\frac{1}{2}\left(\rho_{f}\int_{\mathcal{F}_{h^0}}\vert\ug^0\vert^{2} \dd \xg+ \rho_{s}\int_{0}^{L}\vert\eta^1\vert^{2}\dd x+\beta\int_{0}^{L}\vert\partial_{x}\eta^0\vert^{2}\dd x+\alpha\int_{0}^{L}\vert\partial_{xx}\eta^0\vert^{2}\dd x\right).
\end{aligned} 
\]
Since $\varepsilon$ is arbitrary, we obtain the expected energy inequality.
}

\medskip

We  now  show that  any limit $(\overline{\ug},\eta)$ also  satisfies  items $ii)$ and  $iii)$ of Definition
\ref{def.ws}  of weak solutions  under the further assumption that the following lemma holds true:

\begin{lemma}\label{lemma.convergence.L2} Up to the extraction of a subsequence that we do not relabel,  we have $(\rho_\gamma\overline{\ug}_{\gamma},\partial_{t}\eta_{\gamma})\underset{\gamma\rightarrow 0}{\longrightarrow}(\rho\overline{\ug},\partial_{t}\eta)$ strongly in $L^{2}(\CMT)\times L^{2}((0, T)\times (0,L))$. 
\end{lemma}

So, fix $(\wg, d)\in\mathcal{C}^{\infty}(\CMT)\times \mathcal{C}^{\infty}(0, L) $ such that $(\wg(t), d(t))\in\mathcal{X}[h(t)]$ for all $t\in(0,T)$. 
Due to the uniform convergence of $h_\gamma$ and the special structure of $\mathcal{X}[h(t)]$ for which we require that  $\wg=0$ in the neighbourhood of $\mathcal{C}^0_{-1}$  and $\wg\cdot\BS{e}_{1}=0$ in the neighbourhood of $\mathcal{S}_h(t)$, there exists $\gamma_{0}>0$ such that, for all $0<\gamma<\gamma_{0}$, $(\wg(t),d(t))\in X[h_{\gamma}(t)]$ for all $t\in(0,T)$. Hence, $(\wg,d)$ is a test function for any $\gamma$ small enough and for a.e. $t\in(0,T)$, $(\overline{\ug}_{\gamma},\eta_{\gamma})$ satisfies
\begin{equation}\label{weak.formulation.FS.gamma.limit}
\begin{aligned}
&\rho_{f}\int_{\CM}\rho_{\gamma}\overline{\ug}_{\gamma}(t)\cdot\wg(t) -\rho_{f}\int_{0}^{t}\int_{\CM}\rho_{\gamma}\overline{\ug}_{\gamma}\cdot\partial_{t}\wg + (\rho_{\gamma}\overline{\ug}_{\gamma}\cdot\nabla)\wg\cdot\rho_{\gamma}\overline{\ug}_{\gamma}\\
&+\rho_{s}\int_{0}^{L}\partial_{t}\eta_{\gamma}(t)d(t)-\rho_{s}\int_{0}^{t}\int_{0}^{L}\partial_{t}\eta_{\gamma}\partial_{t}d+\mu\int_{0}^{t}\int_{\CM}\rho_{\gamma} \nabla\ug_{\gamma}:\nabla\wg\\
&+\int_{0}^{t}\int_{0}^{L}\beta\partial_{x}\eta_{\gamma}\partial_{x}d + \alpha\partial_{xx}\eta_{\gamma}\partial_{xx}d+\gamma\int_{0}^{t}\int_{0}^{L}\partial_{tx}\eta_{\gamma}\partial_{x}d=\rho_{f}\int_{\CM}\rho_{\gamma}\overline{\ug}^{0}_{\gamma}\cdot\wg(0)+\rho_{s}\int_{0}^{L}\eta^{1}_{\gamma}d(0).
\end{aligned}
\end{equation}

Let us recall that, thanks to Lemma \ref{lemma.convergence.L2}, $\rho_\gamma \ug_{\gamma}$ strongly converges in $L^2(\hat\Omega)$. The convergence of $\rho_{\gamma} \nabla\ug_{\gamma}$ is proved in Lemma \ref{lemma.gradient} and we can pass to the limit all the terms of \eqref{weak.formulation.FS.gamma.limit}. The pair $(\overline{\ug},\eta)$ satisfies, for a.e. $t\in(0,T)$,
\[
\begin{aligned}
&\rho_{f}\int_{\CM}\rho\overline{\ug}(t)\cdot\wg(t) -\rho_{f}\int_{0}^{t}\int_{\CM}\rho\overline{\ug}\cdot\partial_{t}\wg + (\rho\overline{\ug}\cdot\nabla)\wg\cdot\rho\overline{\ug}\\
&+\rho_{s}\int_{0}^{L}\partial_{t}\eta(t)d(t)-\rho_{s}\int_{0}^{t}\int_{0}^{L}\partial_{t}\eta\partial_{t}d+\mu\int_{0}^{t}\int_{\CM}\rho \ug:\nabla\wg\\
&+\int_{0}^{t}\int_{0}^{L}\beta\partial_{x}\eta\partial_{x}d + \alpha\partial_{xx}\eta\partial_{xx}d=\rho_{f}\int_{\CM}\rho\overline{\ug}^{0}\cdot\wg(0)+\rho_{s}\int_{0}^{L}\eta^{1}d(0),
\end{aligned}
\]
which is a rewriting of the variational formulation \eqref{weak.formulation.FS}.  Thus the item $iii)$ of Definition \ref{def.ws} is satisfied.
The last point to verify is the kinematic condition at the fluid--structure interface. We know that 
$$\ug_\gamma(t, x, h_\gamma(t, x))= \partial_t \eta_\gamma(t, x) \BS{e}_2.$$
From Lemma \ref{lemma.convergence.L2}   the right hand side converges strongly in $L^2(0, T; L_\sharp^2(0, L))$ towards $\partial_t \eta$. It converges also weakly in $L^2(0, T; H_\sharp^{1/2}(0, L))$.  The left hand side is the trace of the function $(x, z) \mapsto \ug_\gamma(t, x, (1+h_\gamma(t, x))z + h_\gamma(t, x))$ on $z=1$. Thanks to the previous convergences this function converges strongly in $L^2(\widehat{{\mathcal C}_0^1})$ and weakly in $L^2(0, T; H^1_\sharp({\mathcal C}_0^1))$ towards  $\ug(t, x, (1+h(t, x))z + h(t, x))$. Hence by continuity of the trace we obtain
$$\ug(t, x, h(t, x))= \partial_t \eta(t, x) \BS{e}_2,$$
so that item $ii)$ of Definition \ref{def.ws}, which  completes the proof of Theorem \ref{main.theorem}.

\medskip

It  thus remains  to prove Lemma \ref{lemma.convergence.L2}. As is usual for fluid--structure problems, the sequence of domains is unknown and depends on time and here on the viscosity parameter, so that standard Aubin--Lions lemma cannot be applied directly to obtain compactness of the velocities. One key point is to build a piecewise in time, regular enough in space, interface, dealing with possible contact, close to the sequence of interfaces but always lower. The construction of this artificial interface is the aim of the next subsection. We then conclude  the proof in the final subsection.  Thanks to the variational formulation and to this well chosen interface ``from below" we obtain bounds on the time derivative of an $L^2$ projection of the velocities, for which we are able to apply  an adapted version of the Aubin--Lions lemma. It implies that the sequence of velocities is nearly compact. Next the key idea is to use that the velocities can be approximated, in $H^s$ for some $s>0$, by velocities associated to the interface ``from below"  so that we can ``fill" the gap. This relies on  the continuity properties of the $H^s$--projector operator obtained in Lemma \ref{lem:projector}.

\subsection{Step 2. Construction of an interface from below.}
Before the construction of the interface from below, we analyze a simple method to approximate a given stationary deformation  from below. Namely, given $h \in H^2_{\sharp}(0,L)$ satisfying $h \geq 0$ and $\mu >0,$ we denote
\[
\hb_{\mu} := [h-\mu]_{+}.
\]
where the subscript $+$ denotes here the positive part of functions. The properties of this approximation process
are summarized in the following lemma:
\begin{lemma}\label{lemma.halpha}
Let $h \in H^{2}_{\sharp}(0,L)$ with $h \geq 0$ and $\mu >0.$ 
Then $\hb_{\mu}\in \mathcal{C}_\sharp^{0}(0,L)$ and, given $\kappa \in (0,1/2),$ there exists a constant $C$
independent of $\mu$ and $h$  for which:
\begin{align}
& \norme{\hb_{\mu}}_{H_\sharp^{1+\kappa}(0,L)} +  \norme{\hb_{\mu}}_{W^{1,\infty}_\sharp(0,L)}\leq C\norme{h}_{H_\sharp^{2}(0,L)}, \label{h1}\\
& \norme{\hb_{\mu}-h}_{W_\sharp^{1, \infty}(0,L)}\leq {\mu}+\displaystyle\sup_{\{x\in[0, L]\mid h(x)\leq \mu\}}\vert  h'(x)\vert.\label{h2}
\end{align}
 \end{lemma}
\begin{proof}Let  $\mu>0$ and $h \in H^{2}_{\sharp}(0,L)$ non-negative. The first statement $\hb_{\mu}\in \mathcal{C}_\sharp^{0}(0,L)$ is standard. We prove the two inequalities \eqref{h1}, \eqref{h2} successively.

\medskip

{\bf Step 1: proof of inequality \eqref{h1}.} Since $h \in {\mathcal C}^0_{\sharp}(0,L)$ the subset $\{x\in(0, L), h(x)>\mu\}$ is open. 
We may then construct at most denumerable sets $\{c_i, \; i \in \mathcal I_{\mu}\}$ and $\{ d_i, \; i \in \mathcal I_{\mu}\}$ such that $\{x\in(0, L), h(x)>\mu\}= \bigcup_{i\in\mathcal I_{\mu}}(c_{i},d_{i}).$ The successive derivatives of $[h-\mu]_{+}$ then read
\[
[h-\mu]'_{+}= h'\mathds{1}_{h>\mu}, \quad  [h-\mu]{''}_{+}= h{''}\mathds{1}_{h>\mu}+\sum_{i\in \mathcal I_{\mu}}\delta_{c_{i} }h{'}(c_{i})- \delta_{d_{i}}h'(d_{i}),
\]
To show that $\hb_{\mu} \in H^{1+\kappa}_{\sharp}(0,L)$ and have a bound on its norm, we now prove  that
$[h- {\mu}]{''}_{+} \in H^{\kappa-1}_{\sharp}(0,L).$ 

\medskip

Using the $H^{2}$-regularity of $h$ we obtain that $h{''}\mathds{1}_{h>\mu} \in L_\sharp^{2}(0,L) \subset H^{\kappa-1}_{\sharp}(0,L)$. Moreover  for any test function $\varphi \in \mathcal{D}(0, L)$, we have:
\[
\vert\langle \delta_{c_{i}},\varphi\rangle\vert = \vert \varphi(c_{i})\vert\leq \norme{\varphi}_{\mathcal{C}_\sharp(0,L)}\leq C\norme{\varphi}_{H_\sharp^{1-\kappa}(0,L)},
\]
where $C>0$ stands for the constant associated with the Sobolev embedding $H_\sharp^{1-\kappa}(0,L) \subset \mathcal{C}^0_\sharp(0,L)$. Hence $\delta_{c_{i}} \in H^{\kappa-1}(0,L)$, with a norm independent of $c_{i}$. To show that $[h- {\mu}]{''}_{+} \in H^{\kappa-1}_{\sharp}(0,L)$ it   then remains to prove that the sums 
\[
 \sum_{i\in \mathcal I_{\mu}}\delta_{c_{i} }h{'}(c_{i}), \quad \sum_{i\in \mathcal I_{\mu}}\delta_{d_{i} }h{'}(d_{i}),
\]
do converge normally in the Banach space $H_\sharp^{\kappa-1}(0,L).$

\medskip

Let $i \in \mathcal I_{\mu},$ since $h(c_{i})=h(d_{i})=\mu,$  there exists $b_{i}\in(c_{i},d_{i})$ such that $h'(b_{i})=0$. This implies
\[\vert h'(c_{i})\vert\leq \vert b_{i}-c_{i}\vert^{1/2}\norme{h}_{H^{2}(c_{i},d_{i})} \leq \vert d_{i}-c_{i}\vert^{1/2}\norme{h}_{H^{2}(c_{i},d_{i})},\]
and thus
\begin{equation}\label{regularity.second.derivative}
\begin{aligned}
\sum_{i \in \mathcal I_{\mu}} \norme{h'(c_{i})\delta_{c_{i}}}_{H_\sharp^{\kappa-1}(0,L)}&{}\leq C\sum_{i \in \mathcal I_{\mu}}\vert d_{i}-c_{i}\vert^{1/2}\norme{h}_{H^{2}(c_{i},d_{i})}\\
&\leq C\left(\sum_{i \in \mathcal I_{\mu}} \vert d_{i}-c_{i}\vert\right)^{1/2}\left(\sum_{i \in \mathcal I_{\mu}}^{N}\norme{h}_{H^{2}(c_{i},d_{i})}^{2}\right)^{1/2}\\
&\leq CL^{1/2}\norme{h}_{H_\sharp^{2}(0,L)}.
\end{aligned}
\end{equation}
Consequently,  we have
\[
 \sum_{i\in \mathcal I_{\mu}}\delta_{c_{i} }h{'}(c_{i}) \in H^{\kappa -1}_{\sharp}(0,L)  \textrm{ and } \|\sum_{i\in \mathcal I_{\mu}}\delta_{c_{i} }h{'}(c_{i})\|_{H^{\kappa -1}_{\sharp}(0,L) } \leq  CL^{1/2}\norme{h}_{H_\sharp^{2}(0,L)}.
\]
The argument is similar for $\displaystyle\sum_{i\in \mathcal I_{\mu}}\delta_{d_{i}}h'(d_{i})$. This completes the proof of estimate
\eqref{h1}.

\medskip

{\bf Step 2:  proof of estimate \eqref{h2}.} Since $h\in H_\sharp^2(0, L)$,  thanks to the continuous embedding of  $H_\sharp^1(0, L)$ in  $L_\sharp^\infty(0, L),$ we deduce that $[h-\mu]_{+}\in W_\sharp^{1,\infty}(0, L)$.
Furthermore it is clear that, for any $ x\in [0, L]$,
$$\vert [h-\mu]_{+}(x)-h(x)\vert \leq \mu, \qquad [h-\mu]'_{+}(x)-h'(x) =  -h'(x) \mathds{1}_{h\leq\mu}(x).$$
This implies that \eqref{h2} is satisfied.
\end{proof}

The next lemma ensures that the right-hand side of estimate \eqref{h2}  goes to zero when $\mu$ goes to zero:
\begin{lemma}\label{lemma.conv.zero}Let $h \in \mathcal{C}_\sharp^{1}(0,L)$ with $h\geq 0$. The following limit holds
\[\sup_{\{x\in [0, L]\mid h(x)\leq\mu\}}\vert h'(x)\vert \underset{\mu\rightarrow 0}{\longrightarrow} 0.
\]
\end{lemma}
\begin{proof}Since $ h \in \mathcal{C}^1_\sharp(0,L)$ we have that $\{ x \in [0,L]\mid h(x)\leq\mu\}$ is a compact subset  of $[0,L]$ and that there exists $x_{\mu}\in [0,L]$ such that:
\[
\sup_{\{x\in [0, L]\mid h(x)\leq\mu\}} |h'(x)|= |h'(x_{\mu})|.
\] 
Note that, by construction, we have $h(x_{\mu})\leq \mu$.

\medskip

Using the compactness of $[0,L]$ we have $x_{\mu}\rightarrow \bar x\in [0,L]$ as $\mu$ goes to zero (up to  a subsequence). Using the continuity of $h$ and passing to the limit in the inequality $h(x_{\mu})\leq \mu$ we obtain $h(\bar x)=0$. Moreover, using that $h(x)\geq 0$ we deduce that $\bar x$ is a local minimum of $h$ and thus that $h'(\bar x)=0$. Finally the continuity of $h'$ ensures that $|h'(x_{\mu})|\underset{\mu\rightarrow 0}{\longrightarrow}|h'(\bar x)|=0$.
\end{proof}

We recall that the sequence $(h_{\gamma})_{\gamma >0}$ we consider
 converges to $h$ strongly in $\mathcal{C}^{0}([0,T];\mathcal{C}_\sharp^{1}(0,L))$ with $h\in L^{\infty}(0,T;H_\sharp^{2}(0,L)) \cap W^{1,\infty}(0, T; L^2_\sharp(0, L))$.
We  now are in a position to build a family of approximating interfaces ``from below" of any $h_\gamma$ for $\gamma$ small enough. Namely,  given $\delta >0$ we construct a piecewise-constant (in time) function $\underline{h}_{\delta}$ 
such that there exists $\gamma_0 >0$ for which 
\begin{align}
& \|\hb_{\delta}\|_{L^{\infty}(0,T;W^{1,\infty}_{\sharp}(0,L))} \leq C, \label{estimate.hb}\\
&\hb_{\delta}\leq h_{\gamma},  && \,\forall \gamma\leq\gamma_{0}, \label{eq.interfacebelow}\\
&\| \hb_{\delta}- h_\gamma\|_{L^\infty(0,T; W_\sharp^{1, \infty}(0 , L))}\leq \delta,  && \,\forall \gamma\leq\gamma_{0}. \label{estimation-diffth}
\end{align}
with $C$ depending only on initial data.

So, let fix $\delta >0$ and introduce parameters $\varepsilon >0,N\in \mathbb N$ to be made precise later on. 
We construct our piecewise approximation as follows.  We consider the subdivision of the time interval $[0,T]=\cup_{0\leq k\leq N}I_{k}$ with $I_{k}=[k\Delta t,(k+1)\Delta t)$, $\Delta t=\frac{T}{N+1}$ and we fix $t_{k}\in I_{k}$ such that $\norme{h(t_{k})}_{H^{2}_{\sharp}(0,L)}\leq \norme{h}_{L^{\infty}(0,T;H^{2}_{\sharp}(0,L))}$. On each time interval $I_k$ we then define
\begin{equation}
\label{def-sous-h}
\hb_{\delta}(x, t)= [h-2\varepsilon]_+(x, t_k), \quad t\in I_k.
\end{equation}

\medskip

Estimate \eqref{estimate.hb} follows directly from Lemma \ref{lemma.halpha} and estimate \eqref{est-eta}. 
Now, up to a good choice for the parameters $\varepsilon$ and $N,$ $\hb_{\delta}$ satisfies the two
properties \eqref{eq.interfacebelow}  and \eqref{estimation-diffth} for $\gamma$ small enough. %
%
%
%
%
First, let us prove that $\hb_{\delta}\leq h_{\gamma}$ for all $\gamma$ sufficiently small and $N$ sufficiently large (depending on $\varepsilon$). We recall that by interpolation, we have $h \in \mathcal{C}^{0,\theta}([0,T];\mathcal{C}^{1}_{\sharp}(0,L))$ for $\theta\in(0,\frac{1}{4})$ (see embedding \eqref{inclusion-eta}). Consequently, for any $k \leq N,$ we have
\[
\|h(x,t) - h(x,t_{k})\|_{L^{\infty}_{\sharp}(0,L)} \leq C \Delta t^{\theta}, \quad \forall \, t \in I_k.
\]
Similarly, since $h_{\gamma}$ converges to $h$ in $\mathcal{C}^{0}([0,T];\mathcal{C}_\sharp^{1}(0,L))$ we can find $\gamma_0 >0$ such that, for $\gamma \leq \gamma_0,$
\[
\|h(x,t) - h_{\gamma}(x,t)\|_{W^{1,\infty}_{\sharp}(0,L)} \leq  \varepsilon,  \quad \forall \, t \in (0,T).
\]
Assuming that $N$ is chosen such that (with the  above constant $C$)
\[
C \Delta t^{\theta}=C \left(\dfrac T{N+1} \right)^{\theta} < \varepsilon,
\]
we have, for any $k \leq N$,
\[h_{\gamma}(x,t)\geq h(x,t)-\varepsilon\geq h(x,t_{k})-2\varepsilon,\, \qquad \forall(x,t)\in(0,L)\times I_{k},\,\forall \gamma\leq \gamma_{0}.\]
Taking the positive part in the previous inequality (recall that $h_{\gamma}\geq 0$) we obtain \eqref{eq.interfacebelow}. 

\medskip

We now estimate the difference between $\hb_{\delta}$ and $h_{\gamma}$ for $\gamma \leq \gamma_0.$
Given $k \leq N$ we have, for all $\gamma \leq  \gamma_0$
\[\begin{aligned}
\| \hb_{\delta}- h_\gamma\|_{L^\infty(I_k; W_\sharp^{1, \infty}(0 , L))}&\leq \| \hb_{\delta}- h\|_{L^\infty(I_k; W_\sharp^{1, \infty}(0, L))}
+ \| h- h_\gamma\|_{L^\infty(I_k; W_\sharp^{1, \infty}(0, L))},\\
&\leq \| [h-2\varepsilon]_+(t_k)- h(t_k)\|_{W_\sharp^{1, \infty}(0, L)}
+\| h({t_k})- h\|_{L^\infty(I_k; W_\sharp^{1, \infty}(0, L))}
+ \varepsilon,\\
&\leq \| [h-2\varepsilon]_+(t_k)- h(t_k)\|_{W_\sharp^{1, \infty}(0, L)}
+2\varepsilon.
\end{aligned}\]
Applying Lemma \ref{lemma.halpha}, this entails
$$
\| [h-2\varepsilon]_+(t_k)- h(t_k)\|_{W_\sharp^{1, \infty}(0, L)}\leq 2\varepsilon +\sup_{\{x\in [0, L], h(x,t_k)\leq 2\varepsilon\}}\vert\partial_x h(x,t_k)\vert.
$$
Finally we obtain
\[
\| \hb_{\delta}- h_\gamma\|_{L^\infty((0,T); W_\sharp^{1, \infty}(0 , L))}\leq 4 \varepsilon + \sup_{\{x\in [0, L], h(x, t_k)\leq 2\varepsilon\}}\vert\partial_x h(x, t_k)\vert . 
\]
Consequently, applying Lemma  \ref{lemma.conv.zero} and choosing $\varepsilon>0$ sufficiently small with the corresponding $N \in\Natural$ and $\gamma_{0}>0$ we obtain that the interface $\hb_{\delta}$ satisfies  \eqref{estimation-diffth}.

%

\subsection{Step 3. $L^2$-compactness of the velocities.}

In this section we study the $L^{2}$-convergence of the pair $(\rho_{\gamma}\overline{\ug}_{\gamma},\partial_{t}\eta_{\gamma})$ stated in Lemma \ref{lemma.convergence.L2}. 
We know that, up to a subsequence that we do not relabel,   $\rho_{\gamma}\overline{\ug}_{\gamma}\rightharpoonup\rho\overline{\ug}$ weakly in $L^{2}(0,T;\BS{L}_\sharp^{2}(\CM))$, and $\partial_t\eta_\gamma \rightharpoonup \partial_{t}\eta$ in $L^2(0,T;L^2_{\sharp}(0,L))$. To prove the strong convergence of the sequence $(\rho_{\gamma}\overline{\ug}_{\gamma},\partial_{t}\eta_{\gamma})$ to $(\rho\overline{\ug},\partial_{t}\eta)$ it remains to show that the following convergence holds true:
\begin{equation}\label{convergence.int}
\rho_{f}\int_{0}^{T}\int_{\CM}\vert \rho_{\gamma}\overline{\ug}_{\gamma}\vert^{2}+\rho_{s}\int_{0}^{T}\int_{0}^{L}\vert\partial_{t}\eta_{\gamma}\vert^{2}\underset{\gamma\rightarrow 0}{\longrightarrow}\rho_{f}\int_{0}^{T}\int_{\CM}\vert \rho\overline{\ug}\vert^{2}+\rho_{s}\int_{0}^{T}\int_{0}^{L}\vert\partial_{t}\eta\vert^{2}.
\end{equation}
We recall that we endow $X^0 := \mathbf{L}^2(\Omega) \times L^2_{\sharp}(0,L)$ with the scalar product:
\[
 ((\overline{\vg},\dot{\eta}), (\overline{\wg},d))_{X^0} = \rho_f \int_{\Omega}  \overline{\vg} \cdot \overline{\wg} + \rho_s \int_{0}^L \dot{\eta} d.
\] 
In particular, with these notations, the right hand side of \eqref{convergence.int}   also reads:
\[
\int_0^T ((\rho_{\gamma}\overline{\ug},\partial_t \eta_{\gamma}), (\overline{\ug},\partial_t \eta_{\gamma}))_{X^0}.
\]
By restriction, this bilinear form enables to consider any element of $X^0$ as an element of $(X^s)'$ via the formula
\begin{equation} \label{eq:embedX0}
\langle (\overline{\vg},\dot{\eta}), (\overline{\wg},d) \rangle_{(X^s)',X^s } = ((\overline{\vg},\dot{\eta}), (\overline{\wg},d))_{X^0}
\quad \forall \, ((\overline{\vg},\dot{\eta}),(\overline{\wg},d)) \in X^0 \times X^s.
\end{equation}
In what follows we use this identification without mentioning it.

To obtain \eqref{convergence.int}, we show actually that, up to extract again a denumerable times subsequences, we can prove that the error terms
\[
Err_{\gamma} :=   \rho_{f}\int_{0}^{T}\int_{\CM}\vert \rho_{\gamma}\overline{\ug}_{\gamma}\vert^{2}+\rho_{s}\int_{0}^{T}\int_{0}^{L}\vert\partial_{t}\eta_{\gamma}\vert^{2} - \left( \rho_{f}\int_{0}^{T}\int_{\CM}\vert \rho\overline{\ug}\vert^{2}+\rho_{s}\int_{0}^{T}\int_{0}^{L}\vert\partial_{t}\eta\vert^{2} \right) ,
\]
satisfies $\limsup_{\gamma \to 0} |Err_{\gamma}| \leq \tilde\varepsilon$ for any arbitrary small $\tilde\varepsilon.$ We shall compute $\varepsilon$ with respect to the parameter $\delta >0$ fixing the interface from below $\hb_{\delta}$ satisfying 
\eqref{eq.interfacebelow}--\eqref{estimation-diffth} as in the previous subsection. So let fix such a $\delta >0.$ 
We recall that the related interface $\hb_{\delta}$
is constant on a family of intervals $(I_k)_{0 \leq k \leq N}$ covering $[0,T].$ Below, we denote $\hb_{\delta,k}$ the value
of $\hb_{\delta}$ on  $I_k.$ We then split  the time integral and introduce the projector $\mathbb P^s[\hb_{\delta,k}]$
for a given $s < 1/2.$  This yields:
\begin{align*}
Err_{\gamma} & = \sum_{k=0}^N \int_{I_{k}} ( (\rho_{\gamma} \overline{\ug}_{\gamma},\partial_t \eta_{\gamma}),( \overline{\ug}_{\gamma},\partial_t \eta_{\gamma}))_{X^0} -  ( (\rho \overline{\ug},\partial_t \eta),( \overline{\ug},\partial_t \eta))_{X^0} \\
                      & =  \sum_{k=0}^N \int_{I_{k}} ( (\rho_{\gamma} \overline{\ug}_{\gamma},\partial_t\eta_{\gamma}),( \overline{\ug}_{\gamma},\partial_t \eta_{\gamma})-  \mathbb{P}^{s}[\hb_{\delta, k}](\overline{\ug}_{\gamma},\partial_{t}\eta_{\gamma}))_{X^0} \\
                      & \qquad + \int_{I_{k}} (  \mathbb{P}^{s}[\hb_{\delta, k}](\overline{\ug}_{\gamma},\partial_{t}\eta_{\gamma}),(\rho_{\gamma} \overline{\ug}_{\gamma},\partial_t \eta_{\gamma}))_{X^0}  \\
                      & \qquad - \int_{I_{k}} ( (\rho \overline{\ug} ,\partial_t \eta),( \overline{\ug},\partial_t \eta)- \mathbb{P}^{s}[\hb_{\delta, k}](\overline{\ug},\partial_{t}\eta),\partial_t \eta))_{X^0} \\
                      & \qquad - \int_{I_{k}} ( \mathbb{P}^{s}[\hb_{\delta, k}](\overline{\ug},\partial_{t}\eta),(\rho \overline{\ug},\partial_t \eta))_{X^0} .
\end{align*}
Then, since $\mathbb P^s[\hb_{\delta,k}](\overline{\ug}_\gamma,\partial_t \eta_\gamma) \in X^s[\hb_{\delta,k}]\subset X[\hb_{\delta,k}] $ and  thanks the identification \eqref{eq:embedX0}, we write
\begin{align*}
(  \mathbb{P}^{s}[\hb_{\delta, k}](\overline{\ug}_{\gamma},\partial_{t}\eta_{\gamma}),(\rho_{\gamma} \overline{\ug}_{\gamma},\partial_t \eta_{\gamma}))_{X^0}  
& =  (  \mathbb{P}^{s}[\hb_{\delta, k}](\overline{\ug}_{\gamma},\partial_{t}\eta_{\gamma}),\mathbb P[\hb_{\delta,k}](\rho_{\gamma} \overline{\ug}_{\gamma},\partial_t \eta_{\gamma}))_{X^0}  \\
& = \langle \mathbb P[\hb_{\delta,k}](\rho_{\gamma} \overline{\ug}_{\gamma},\partial_t \eta_{\gamma}) ,  \mathbb{P}^{s}[\hb_{\delta, k}](\overline{\ug}_{\gamma},\partial_{t}\eta_{\gamma}) \rangle_{(X^s)',X^s}.
\end{align*}
Proceeding similarly with the limit term, we obtain the following splitting
\[
Err_{\gamma} = \sum_{k=0}^N Err^{app}_{\gamma,k} + Err^{conv}_{\gamma,k} - Err^{app}_{k},
\]
where, for arbitrary $k \leq N,$ we denote
\begin{align*}
Err^{app}_{\gamma,k}  & = \int_{I_{k}} \Bigl( (\rho_{\gamma} \overline{\ug}_{\gamma},\partial_t \eta_{\gamma}),( \overline{\ug}_{\gamma},\partial_t \eta_{\gamma})-  \mathbb{P}^{s}[\hb_{\delta, k}](\overline{\ug}_{\gamma},\partial_{t}\eta_{\gamma})\Bigr)_{X^0} \\
Err^{app}_{k} &=  \int_{I_{k}} \Bigl( (\rho \overline{\ug} ,\partial_t \eta),( \overline{\ug},\partial_t \eta)- \mathbb{P}^{s}[\hb_{\delta, k}](\overline{\ug},\partial_{t}\eta),\partial_t \eta)\Bigr)_{X^0}\\  
Err^{conv}_{\gamma,k} &= \int_{I_k} \langle \mathbb P[\hb_{\delta,k}](\rho_{\gamma} \overline{\ug}_{\gamma},\partial_t \eta_{\gamma}) ,  \mathbb{P}^{s}[\hb_{\delta, k}](\overline{\ug}_{\gamma},\partial_{t}\eta_{\gamma}) \rangle_{(X^s)',X^s}
 - \langle \mathbb P[\hb_{\delta,k}](\rho \overline{\ug} ,\partial_t \eta) ,  \mathbb{P}^{s}[\hb_{\delta, k}](\overline{\ug},\partial_{t}\eta) \rangle_{(X^s)',X^s}.
\end{align*}

\medskip

For the two first type of terms we use the fact that projection on $X^s[h_{\delta, k}]$ has good approximation properties. So, to estimate the error terms $Err^{app}_{\gamma,k}$, we use Lemma \ref{lem:projector}   for $\kappa=1/4$. Indeed, from the bound \eqref{bound-u-2} and Lemma \eqref{lem:extension-u}, we know that $(\overline{\ug}_{\gamma},\partial_t \eta_\gamma)$  satisfies, for a.e. $t \in I_k$,  $(\overline{\ug}_{\gamma}(t),\partial_t \eta_{\gamma}(t)) \in X^s[h_{\gamma}],$ for $s<1/2$, with $\overline{\ug}_{\gamma} \in H_\sharp^1(\mathcal F^-_{h_{\gamma}}).$  Moreover we remark that both interfaces  $h_{\gamma}$ and $h_\delta$ belong  to $H^{1+\kappa}_{\sharp}(0,L) \cap W_\sharp^{1,\infty}(0,L)$ and that, thanks to the definition of $h_\delta$, there exists   $A >0$ independent of $\gamma$ and $\delta$,  such that
\begin{align*}
&\norme{h_{\gamma}}_{L^{\infty}(0,T;H^{1+\kappa}_{\sharp}(0,L))}+\norme{h_{\gamma}}_{L^{\infty}(0,T;W^{1,\infty}_{\sharp}(0,L))}+\norme{\hb_{\delta}}_{L^{\infty}(0,T;H^{1+\kappa}_{\sharp}(0,L))}+\norme{\hb_{\delta}}_{L^{\infty}(0,T;W^{1,\infty}_{\sharp}(0,L))}\leq A.
\end{align*}
Finally by construction $h_\delta$ and $h_\gamma$ are close in $W^{1, \infty}_\sharp(0, L)$ and \eqref{estimation-diffth} is satisfied.
Hence,
Lemma \ref{lem:projector} implies, for $s<1/8$
\begin{align*}
\norme{\mathbb{P}^{s}[\hb_{\delta, k}](\overline{\ug}_{\gamma}(\cdot,t),\partial_{t}\eta_{\gamma}(\cdot,t)) -(\overline{\ug}_{\gamma}(\cdot,t),\partial_{t}\eta_{\gamma}(\cdot,t))}_{X^0}
&\leq\norme{\mathbb{P}^{s}[\hb_{\delta, k}](\overline{\ug}_{\gamma}(\cdot,t),\partial_{t}\eta_{\gamma}(\cdot,t)) -(\overline{\ug}_{\gamma}(\cdot,t),\partial_{t}\eta_{\gamma}(\cdot,t))}_{X^s}\\
&\leq C_{A}(\delta)\norme{\nabla \overline{\ug}_{\gamma}(\cdot,t)}_{\BS{L}^{2}_{\sharp}(\mathcal{F}^{-}_{h_{\gamma}(t)})},
\end{align*}
with $\lim_{x\rightarrow 0} C_A(x)=0$.
Using Cauchy--Schwartz inequality, we deduce from the previous estimate that
\begin{align*}
\sum_{k=0}^N |Err_{\gamma,k}^{app}| & \leq \sum_{k=0}^N \int_{I_k} \norme{\mathbb{P}^{s}[\hb_{\delta, k}](\overline{\ug}_{\gamma}(\cdot,t),\partial_{t}\eta_{\gamma}(\cdot,t)) -(\overline{\ug}_{\gamma}(\cdot,t),\partial_{t}\eta_{\gamma}(\cdot,t))}_{X^0}
\|(\overline{\ug}_{\gamma},\partial_t \eta_{\gamma})\|_{X^0} \\ 
& \leq C_A({\delta})\|(\overline{\ug}_{\gamma},\partial_t \eta_{\gamma})\|_{L^\infty(0, T;X^0)} \int_0^T  \|\nabla \overline{\ug}_{\gamma}\|_{L^2(\mathcal F_{h_{\gamma}(t)})} .
\end{align*}
Then we use the uniform estimates \eqref{bound-u-1}, \eqref{bound-u-2} to obtain that there exists a constant $C_1$
(depending only on initial data and $T$) such that, for $\gamma \leq \gamma_0$
\begin{equation} \label{eq_Err1}
\sum_{k=0}^N |Err_{\gamma,k}^{app}| \leq C_1 C_A(\delta).
\end{equation}
 Similarly  we have
\begin{equation} \label{eq_Err2}
\sum_{k=0}^N |Err_{k}^{app}| \leq C_1 C_A(\delta).
\end{equation}
To complete the proof, the following term remains to be estimated
\[
\limsup_{\gamma \to 0} \sum_{k=0}^N |Err_{\gamma,k}^{conv}|.
\] 
At first,  we prove that, for a fixed $k \leq N$ and  up to a subsequence,  $\mathbb{P}[\hb_{\delta, k}](\rho_{\gamma}\overline{\ug}_{\gamma},\partial_{t}\eta_{\gamma})$ converges strongly to $\mathbb{P}[\hb_{\delta, k}](\rho\overline{\ug},\partial_{t}\eta)$ in $L^{2}(I_{k};(X^{s}[\hb_{\delta, k}])')$. Note that, since $(\rho_{\gamma}\overline{\ug}_{\gamma},\partial_t \eta_{\gamma})$ converges weakly to $(\rho \overline{\ug},\partial_t \eta)$ in $L^2(I_k;X^0)$ the only difficulty relies on showing that the sequence $\mathbb{P}[\hb_{\delta, k}](\rho\overline{\ug},\partial_{t}\eta)$ is relatively compact in $L^{2}(I_{k};(X^{s}[\hb_{\delta, k}])').$
To do so we apply an adapted version of Aubin--Lions lemma that can be found in \cite[Section 4.3]{Fujita-Sauer} that reads
\begin{lemma} Let us consider three Hilbert spaces $M_i$, $i=1, 2,3 $ and two operators $T : M_0 \mapsto M_1$ and $S : M_0 \mapsto M_2$ satisfying
\begin{itemize}
\item $T$ and $S$ are two linear compact operators,
\item  $Su=0$ implies $Tu=0$.
\end{itemize}
If $(u_n)$ is bounded in $L^2(0, T; M_0)$ and $(\partial_t S u_n)$ is bounded in $L^2(0, T; M_2)$, then $T u_n$ is a compact set of $L^2(0, T; M_1)$.
\end{lemma}

We are going to use this version of Aubin--Lions lemma with the triplet $(X^0,(X^{s}[\hb_{\delta, k}])',(X^{1}[\hb_{\delta, k}])')$ and with $S= i_1\circ \mathbb{P}$, $T=i_s\circ  \mathbb{P}$, where $i_l$ denotes the injection of $X[\hb_{\delta, k}]$ into $(X^{l}[\hb_{\delta, k}])'$.  The mappings $i_l$ are indeed injective functions since Lemma \ref{lemma.density} implies that the continuous embedding $X^l[\hb_{\delta,k}] \subset X[\hb_{\delta, k}]$ is dense, for $l>0$ and these densities imply that $X[\hb_{\delta, k}] $ is continuously embedded in $(X^{l}[\hb_{\delta, k}])'$, for $l>0$. Moreover, thanks to Rellich--Kondrachov theorem the embeddings $X^l[\hb_{\delta,k}] \subset X[\hb_{\delta, k}]$ are compact.   The dual of a compact operator  being compact,  $X[\hb_{\delta, k}] $ is compactly embedded in  $(X^{l}[\hb_{\delta, k}])'$ for $l>0$. Consequently $i_1\circ \mathbb{P}$ and $i_s\circ  \mathbb{P}$ are compact linear operators. Moreover $i_1\circ \mathbb{P}(\wg, b)=0$ implies $\mathbb{P}(\wg, b)=0$ so that the second point is clearly satisfied.

Next applying \eqref{bound-u-1}, we have that $(\overline{\ug}_{\gamma},\partial_{t}\eta_{\gamma})$ is uniformly bounded in $\gamma$ in $L^{2}(0,T;X^0)$. Thus the sequence $(\rho_{\gamma}\overline{\ug}_{\gamma},\partial_{t}\eta_{\gamma})$ is  bounded in $\gamma$ in $L^{2}(I_{k};X^0)$. We must now obtain a uniform bound for
$\partial_{t}\mathbb{P}[\hb_{\delta, k}](\rho_{\gamma}\overline{\ug}_{\gamma},\partial_{t}\eta_{\gamma})$ in $L^p(I_k; (X^1[\hb_{\delta,k}])').$ Precisely, we look for an estimate of the type
\[
\left| -\int_{I_{k}}\left( \mathbb{P}[\hb_{\delta, k}](\rho_{\gamma}\overline{\ug}_{\gamma},\partial_{t}\eta_{\gamma}), \partial_{t}(\wg,b)     \right)_{X^0}
\right| \leq C \int_{I_k} \|(\wg(t),d(t))\|_{X^1[\hb_{\delta,k}]}^{2}
{\rm d}t,
\quad \forall \, (\wg,d) \in L^2(I_k; X^1[\hb_{\delta,k}]).
\]
To obtain such an estimate we use the variational formulation \eqref{weak.formulation.FS} satisfied by $(\overline{\ug}_{\gamma},\partial_{t}\eta_{\gamma})$. We consider
$(\wg,d) \in \mathcal C^{\infty}_c(I_k; X^1[\hb_{\delta,k}])$. This is an admissible test function since $\hb_{\delta}\leq h_\gamma$ and since, in the case where $\gamma>0$ for which $\min_{x\in [0, L]}h_\gamma(t, x)>0,\forall t\in [0, T]$, we can consider test functions in $C^{\infty}([0, T]; X^1[h_\gamma(t)])$.
 We obtain
\[\begin{aligned}
-\int_{I_{k}}\left( \mathbb{P}[\hb_{\delta, k}](\rho_{\gamma}\overline{\ug}_{\gamma},\partial_{t}\eta_{\gamma}), \partial_{t}(\wg,b)     \right)_{X^0} &=-\rho_{f}\int_{I_{k}}\rho_{\gamma}\overline{\ug}_{\gamma}\cdot\partial_{t}\wg - \rho_{s}\int_{I_{k}}\partial_{t}\eta_{\gamma}\partial_{t}b\\
&=\rho_{f}\int_{I_{k}}\int_{\mathcal{F}_{h_\gamma(t)}}({\ug}_{\gamma}\cdot\nabla)\wg\cdot{\ug}_{\gamma}-\mu\int_{I_{k}}\int_{\mathcal{F}_{h_\gamma(t)}}\nabla{\ug_\gamma}:\nabla \wg\\
& \qquad -\beta\int_{I_{k}}\int_{0}^{L}\partial_{x}\eta_{\gamma}\partial_{x}b+ \alpha\int_{I_{k}}\int_{0}^{L}\partial_{xx}\eta_{\gamma}\partial_{xx}b \\
& \qquad +\gamma\int_{I_{k}}\int_{0}^{L}\partial_{tx}\eta_{\gamma}\partial_{x}b.
\end{aligned}
\]
The nonlinear convection term is estimated using the $L^{4}$--regularity of $\overline{\ug}_{\gamma}$ stated in Lemma \ref{lem:extension-u}, 
\[\begin{aligned}
\left\vert\rho_{f} \int_{I_{k}}\int_{\CM}(\rho_{\gamma}\overline{\ug}_{\gamma}\cdot\nabla)\wg\cdot\rho_{\gamma}\overline{\ug}_{\gamma}\right\vert&{}\leq \rho_{f}\int_{I_{k}}\norme{\overline{\ug}_{\gamma}(t)}_{\BS{L}^{4}(\CM)}^{2}\norme{\nabla\wg(t)}_{\BS{L}^{2}(\CM)}\dd t\\
&\leq \rho_{f}\norme{\overline{\ug}_{\gamma}}_{L^{4}(\CMT)}^{2}\norme{\nabla\wg}_{L^{2}(I_{k};\BS{L}^{2}(\CM))}.
\end{aligned}
\]
The other terms are estimated directly and we obtain
\[
\left\vert \int_{I_{k}}\left( \mathbb{P}[\hb_{\delta, k}](\rho_{\gamma}\overline{\ug}_{\gamma},\partial_{t}\eta_{\gamma}), \partial_{t}(\wg,b)      \right)_{X^0}\right\vert \leq C\norme{(\wg,b)}_{L^{2}(I_{k};X^{1}[\hb_{\delta, k}])},
\]
where $C$ depends only on the initial data.
The previous inequality implies that $\partial_{t}\mathbb{P}[\hb_{\delta, k}](\rho_{\gamma}\overline{\ug}_{\gamma},\partial_{t}\eta_{\gamma})$ is bounded in $\gamma$ in $L^{2}(I_{k};(X^{1}[\hb_{\delta, k}])')$. It then follows  from the adapted version of  Aubin--Lions lemma that $\mathbb{P}[\hb_{\delta, k}](\rho_{\gamma}\overline{\ug}_{\gamma},\partial_{t}\eta_{\gamma})$ is compact in $L^{2}(I_{k};(X^{s}[\hb_{\delta, k}])')$. Moreover, since $(\overline{\ug}_{\gamma},\partial_t \eta_{\gamma})$ converges weakly to 
$(\overline{\ug},\partial_t \eta)$ in $L^2(I_k; X^s)$, for $s<1/2$ we   also have that 
$\mathbb P^s[\hb_{\delta,k}](\overline{\ug}_{\gamma},\partial_t \eta_{\gamma})$ converges weakly to $\mathbb P^s[\hb_{\delta,k}](\overline{\ug},\partial_t \eta)$ in $L^2(I_k;X^s).$ Combining a strong and a weak convergence result leads to
\begin{equation} \label{eq_Err3}
\lim_{\gamma \to 0} Err_{\gamma,k}^{conv} = 0 \quad \forall \, k \leq N. 
\end{equation}
Finally, combining \eqref{eq_Err1}--\eqref{eq_Err3}, we conclud that 
\[
\limsup_{\gamma \to 0} |Err_{\gamma} | \leq C_1C_A(\delta) 
\] 
for arbitrary $\delta >0.$ We conclude the proof by remarking that $C_A(\delta) \to 0$ when $\delta \to 0.$ For completeness, we remark that in the computations of bounds for $Err_{\gamma}$ we only extract subsequences when we apply the Aubin--Lions lemma. Since we perform extraction a finite number of times for any value of the parameter $\delta$ that we can choose in a denumerable sets ({\em i.e.} a sequence converging to $0$), our proof induces indeed  denumerable extractions of subsequences.


\begin{remark}
In the final weak formulation we consider fluid  test functions that vanishes in the neighbourhood of the bottom of the fluid cavity and that are only transverse in the neighbourhood  of the interface. Note that we could have also consider  fluid test functions that vanishes in the neighbourhood any contact point. It imposes in particular the structure test function to be zero near  each contact point so that they depend implicitly on the solution.  
\end{remark}

\appendix
\section{Proof of lemma \ref{lemma.density}}
\label{Annexe}
This appendix is devoted to a density lemma in the space $X[h].$ 
We first recall the statement of the lemma to be proven and proceed
to the proof.

\begin{lemma}\label{lemma.density.app}
Let $h \in \mathcal C^0_{\sharp}(0,L)$ satisfy $0 \leq h(x) \leq M, \forall x\in [0, L]$ The embedding $X[h]\cap (\mathcal{C}^{\infty}(\overline{\CM}) \times \mathcal C^{\infty}_{\sharp}(0,L)) \subset X[h]$ is dense.
\end{lemma}

\begin{proof}
First notice that the main difficulty here comes from the potential contact {\em i.e.} the points where $h$ is equal to zero. If there is no contact we may construct explicitly a smooth approximating sequence of any pair in $X[h]$ by adapting the arguments of \cite{Chambolle-etal}, see also the construction of approximate initial data in Section \ref{sec:proof1}.

\medskip

When $h$ vanishes, we propose an alternative proof:  in this case we obtain  that $(X[h]\cap \mathcal{C}^{\infty}(\overline{\CM}) \times \mathcal C^{\infty}_{\sharp}(0,L))^{\perp}=\{(\BS{0},0)\}$. So, let $(\ug,\overset{\cdot}{\eta}) \in X[h]$ and assume it satisfies:
\begin{equation}\label{PS.1}\rho_{f}\int_{\CM}\ug\cdot\wg +\rho_{s}\int_{0}^{L}\overset{\cdot}{\eta}d=0, \quad 
\forall \, (\wg,d)\in X[h]\cap \mathcal{C}^{\infty}(\overline{\CM}) \times \mathcal C^{\infty}_{\sharp}(0,L). 
\end{equation}
Using Lemma \ref{lemma.X[h]} there exists $\Psi\in H^{1}_{\sharp}(\CM)$ and $b\in H^{1}_{\sharp}(0,L)$ such that $\ug=\nabla^{\bot}\Psi$ with $\Psi=b(x)$ in $\mathcal{S}_{h}$ and $\Psi=0$ in $\mathcal{C}^{0}_{-1} \cup  I^{c}\times(-1,2M)$ where $\displaystyle I=\{x\in[0,L]\mid h(x)>0\}$. To complete the proof, we obtain that $\ug$ vanishes in $I\times(-1,2M)$. 
Since $I$ is an open subset of $(0,L)$ we may construct an at most denumerable $\mathcal I$ such that 
\[
I=\bigsqcup_{i\in\mathcal I}(a_{i},b_{i})
\] 
where the $(a_{i},b_{i})$ are the connected components of $I$. It is now sufficient to prove that, for arbitrary  $i\in \mathcal I$ there holds $\ug=0$ in $\CM_{i}=(a_{i},b_{i})\times(-1,M)$. This is obtained by a suitable choice of functions $(\wg,d)$ in \eqref{PS.1}.

\medskip

Let fix $i \in \mathcal I$  and $\varepsilon>0$ small enough. Consider $\chi_{\varepsilon} \in \mathcal{C}^{\infty}_{c}((a_{i},b_{i}))$ such that 
\[
\begin{aligned}
& \chi_{\varepsilon}=1 \text{ on $[a_{i}+\varepsilon,b_{i}-\varepsilon]$}, 
\quad  
\text{supp}(\chi_{\varepsilon})\subset[a_{i}+\frac{\varepsilon}{2},b_{i}-\frac{\varepsilon}{2}], \\
&  \|\chi_{\varepsilon}'\|_{L^{\infty}(\mathbb R)}<\frac{1}{\varepsilon} 
\qquad  \|\chi_{\varepsilon}''\|_{L^{\infty}(\mathbb R)}<\frac{1}{\varepsilon^{2}}.
\end{aligned}
\] 
Existence of such a truncation function is classical. 
We now introduce  
\[
\wg_\varepsilon=\nabla^{\perp}(\chi_{\varepsilon}\Psi) \quad 
d_{\varepsilon} = \partial_x (\chi_{\varepsilon}b).
\]
It is straightforward that $(\wg_{\varepsilon},d_{\varepsilon}) \in \BS L^2((a_{i},b_{i}) \times (-1,M)) \times L^2_0((a_{i},b_{i}))$ and has support in $(a_i+\varepsilon,b_i-\varepsilon).$ On the other hand, there exists $\delta_{\varepsilon}>0$ such that $h(x)\geq \delta_{\varepsilon}$
on $(a_i+\varepsilon/2,b_i-\varepsilon/2)$. Setting $h_{\varepsilon} = \max(h,\delta_{\varepsilon})$ we have then that $h_{\varepsilon} \in \mathcal C^0_{\sharp}(0,L)$ does not vanish and $(\wg_{\varepsilon},d_{\varepsilon}) \in X[h_{\varepsilon}].$ Consequently, we may reproduce the arguments in the case of a non vanishing deformation to approximate $(\wg_{\varepsilon},d_{\varepsilon})$ by a sequence of pairs 
in $X[h_{\varepsilon}]  \cap (\mathcal C^{\infty}(\overline{\Omega}) \times \mathcal C^{\infty}_{\sharp}(0,L)).$ Moreover, we emphasize that,
by construction, this sequence has support in $(a_i+\varepsilon/2,b_i-\varepsilon/2)\times (-1,2M)$ also so that it is actually a sequence of $X[h] \cap (\mathcal C^{\infty}(\overline{\Omega}) \times \mathcal C^{\infty}_{\sharp}(0,L))$ that approximates $(\wg_{\varepsilon},d_{\varepsilon})$ in $X[h]$ also.
 
 \medskip
 
Consequently the identity \eqref{PS.1} holds true for $(\wg_{\varepsilon},d_{\varepsilon})$ also and we have
 $$
\int_{\CM_{i}}\ug\cdot\nabla^{\perp}(\chi_{\varepsilon}\Psi) + \int_{a_{i}}^{b_{i}}\overset{\cdot}{\eta} \partial_x (\chi_{\varepsilon} b)=0.
$$
But, recalling that $\nabla^{\bot} \Psi = \ug$ and $\partial_x b = \dot{\eta}$ we may expand the differential operators to yield that:
$$
0 = \int_{\CM_{i}}\ug\cdot\nabla^{\perp}(\chi_{\varepsilon}\Psi) + \int_{a_{i}}^{b_{i}}\overset{\cdot}{\eta}z_\varepsilon=\int_{\CM_{i}}\vert\ug\vert^{2}\chi_\varepsilon + \int_{a_{i}}^{b_{i}}\vert\overset{\cdot}{\eta}\vert^{2}\chi_\varepsilon + \int_{\CM_{i}}\ug\cdot\nabla^{\perp}(\chi_{\varepsilon})\Psi + \int_{a_{i}}^{b_{i}}\overset{\cdot}{\eta}b{\chi}_{\varepsilon}'.
$$
Since $\chi_{\varepsilon}$ depends on the $x$-variable only and $\chi'_{\varepsilon}$ vanishes on $\{a_i,a_i+ \varepsilon,b_i-\varepsilon,b_i\}$  , we have, by integrating by parts:
\begin{equation}\label{calcul-1}
-\int_{\CM_{i}}\ug\cdot\nabla^{\perp}(\chi_{\varepsilon})\Psi=\int_{a_i}^{a_i+\varepsilon} \int_{-1}^{2M} \frac{\Psi^{2}}{2}{\chi}_{\varepsilon}''
+ \int_{b_i-\varepsilon}^{b_i} \int_{-1}^{2M} \frac{\Psi^{2}}{2}{\chi}_{\varepsilon}''
\end{equation}
Similarly we prove the following equality:
\begin{equation}\label{calcul-2}-\int_{a_{i}}^{b_{i}}\overset{\cdot}{\eta}\,b(x) {\chi}_{\varepsilon}'=
\int_{a_i}^{a_i+\varepsilon}\frac{b(x)^{2}}{2}{{\chi}_{\varepsilon}''}
+\int_{b_i-\varepsilon}^{b_i}\frac{b(x)^{2}}{2}{{\chi}_{\varepsilon}''}.
\end{equation}
Since  $\Psi=0$ on $\{a_{i}\}\times(-1,2M)$
a standard Poincar\'e inequalities  entails that:
\[\begin{aligned}
&\int_{a_{i}}^{a_i+\varepsilon} \int_{-1}^{2M} \frac{\Psi^{2}}{2}\leq \frac{\varepsilon^{2}}{4}\int_{a_{i}}^{a_i+\varepsilon} \int_{-1}^{2M}\vert\nabla\Psi\vert^{2}=\frac{\varepsilon^{2}}{4}\int_{a_{i}}^{a_i+\varepsilon}\vert\ug\vert^{2},\\
&\int_{a_{i}}^{a_i+\varepsilon}\frac{b(x)^{2}}{2}\leq\frac{\varepsilon^{2}}{4}\int_{a_{i}}^{a_i+\varepsilon}\left|\partial_{x}b(x)\right|^{2}=\frac{\varepsilon^{2}}{4}\int_{a_{i}}^{a_i+\varepsilon}\vert\overset{\cdot}{\eta}\vert^{2},
\end{aligned}
\]
We have a similar identity for integrals involving $(b_i-\varepsilon,b_i)$ by using that $\Psi = 0$ on $\{b_i\} \times \{-1,2M\}.$ Using finally that $L^{\infty}$-estimate on ${\chi}''_{\varepsilon}$ in \eqref{calcul-1}-\eqref{calcul-2} we conclude
\begin{equation}\label{calcul-3}\int_{a_i+\varepsilon}^{b_i-\varepsilon} \int_{-1}^{2M} \vert\ug\vert^{2}+\int_{a_{i}+\varepsilon}^{b_{i}-\varepsilon}\vert\overset{\cdot}{\eta}\vert^{2}\leq \frac{1}{4}\left(\int_{a_i}^{a_i+\varepsilon} \int_{-1}^{2M} \vert\ug\vert^{2}+\int_{a_i}^{a_i+\varepsilon}\vert\overset{\cdot}{\eta}\vert^{2}
+\int_{b_i-\varepsilon}^{b_i} \int_{-1}^{2M} \vert\ug\vert^{2}+\int_{b_i-\varepsilon}^{b_i}\vert\overset{\cdot}{\eta}\vert^{2} \right).
\end{equation}
Since $(\ug,\dot{\eta})$ are both $L^2$-functions, the right-hand side of this identity vanishes when $\varepsilon \to 0.$ So, letting $\varepsilon\rightarrow 0$ we obtain $(\ug,\overset{\cdot}{\eta})=(\BS{0},0)$ in $\Omega_i$.  This ends the proof.
\end{proof}

\end{document}